\documentclass[11pt,reqno,a4paper,oneside]{amsart}
\usepackage{geometry}
\usepackage[citecolor=black, urlcolor=black, linkcolor=black, colorlinks=true, bookmarksopen=true]{hyperref}
\usepackage{amsmath,amssymb,amsthm,amsfonts,upgreek,bm} 
\usepackage[usenames,dvipsnames]{color}
\usepackage{url}
\usepackage{colonequals}
\usepackage{upgreek}
\usepackage{enumerate}
\usepackage[active]{srcltx}
\usepackage[bbgreekl]{mathbbol}
\newtheorem{theorem}{Theorem}[section]
\newtheorem{corollary}[theorem]{Corollary}

\newtheorem{lemma}[theorem]{Lemma}
\newtheorem{proposition}[theorem]{Proposition}
\newtheorem{assume}{Assumption}

\theoremstyle{definition}
\newtheorem{definition}[theorem]{Definition}
\newtheorem{remark}[theorem]{Remark}
\newtheorem*{notation}{Notation}
\numberwithin{equation}{section}

\DeclareMathOperator{\AC}{AC}
\DeclareMathOperator{\CE}{CE}
\DeclareMathOperator{\CCEE}{\mathbb{CE}}

\DeclareMathOperator{\Expect}{E}

\DeclareMathOperator{\loc}{loc}
\DeclareMathOperator{\var}{var}
\DeclareMathOperator{\Div}{div}
\newcommand\bra[1]{\left({#1}\right)}

\newcommand\set[1]{\left\{#1\right\}}
\newcommand\abs[1]{\left|#1\right|}
\newcommand\skp[2]{\left\langle  #1 , #2 \right\rangle}

\newcommand{\vecfield}{\vec}

\newcommand{\doititle}[1]{\emph{#1}}
\newcommand{\arXiv}[1]{\emph{arXiv: #1}}
\newcommand{\R}{\mathbf{R}}

\newcommand{\N}{\mathbf{N}}

\newcommand{\eps}{\varepsilon}

\newcommand{\cA}{\mathcal{A}}

\newcommand{\cF}{\mathcal{F}}
\newcommand{\cH}{\mathcal{H}}
\newcommand{\cK}{\mathcal{K}}
\newcommand{\cI}{\mathcal{I}}
\newcommand{\cJ}{\mathcal{J}}
\newcommand{\cP}{\mathcal{P}}

\newcommand{\cT}{\mathcal{T}}
\newcommand{\cW}{\mathcal{W}}
\newcommand{\cX}{\mathcal{X}}
\newcommand{\be}{\bm{e}}

\newcommand{\bS}{\bm{S}}
\newcommand{\bx}{\bm{x}}
\newcommand{\bX}{\bm{X}}
\newcommand{\by}{\bm{y}}

\newcommand{\EX}{\mathbb{E}}
\renewcommand{\AA}{\mathbb{A}}
\newcommand{\CC}{\mathbb{C}}
\newcommand{\FF}{\mathbb{F}}
\newcommand{\GG}{\mathbb{G}}

\newcommand{\II}{\mathbb{I}}
\newcommand{\JJ}{\mathbb{J}}

\newcommand{\MM}{\mathbb{M}}
\newcommand{\NN}{\mathbb{N}}

\newcommand{\PPsi}{\mathbb{\Psi}}

\renewcommand{\SS}{\mathbb{S}}
\newcommand{\VV}{\mathbb{V}}
\newcommand{\vv}{\mathbb{v}}
\newcommand{\WW}{\mathbb{W}}

\date{\today}
\title[Discrete mean-field gradient flows]{Gradient flow structure for McKean-Vlasov equations on discrete spaces}


\author{Matthias Erbar}
\address{University of Bonn, Germany}
\email{erbar@iam.uni-bonn.de}
\author{Max Fathi}
\address{University of California, Berkeley}
\email{maxf@berkeley.edu}
\author{Vaios Laschos}
\address{Weierstrass Institute}	
\email{Vaios.laschos@wias-berlin.de}
\author{Andr\'e Schlichting}
\address{University of Bonn, Germany}
\email{schlichting@iam.uni-bonn.de}

\begin{document}

\begin{abstract} In this work, we show that a family of non-linear
  mean-field equations on discrete spaces can be viewed as a gradient
  flow of a natural free energy functional with respect to a certain
  metric structure we make explicit. We also prove that this gradient
  flow structure arises as the limit of the gradient flow structures
  of a natural sequence of $N$-particle dynamics, as $N$ goes to
  infinity.
\end{abstract}

\maketitle

\section{Introduction}
In this work, we are interested in the gradient flow structure of
McKean-Vlasov equations on finite discrete spaces.
They take the form
\begin{equation}\label{eq_evol_nonlin}
\dot{{c}}(t)={c}(t)Q({c}(t))
\end{equation}
where ${c}(t)$ is a flow of probability measures on a fixed finite set
$\mathcal{X}=\{1,\dots,d\}$, and $Q_{xy}(\mu)$ is collection of
parametrized transition rates, that is for each $\mu\in\cP(\cX)$,
$Q(\mu)$ is a Markov transition kernel.

Such non-linear equations arise naturally as the scaling limit for the
evolution of the empirical measure of a system of $N$ particles
undergoing a linear Markov dynamics with mean field interaction. Here
the interaction is of mean field type if the transition rate
$Q^N_{i;x,y}$ for the $i$-th particle to jump from site $x$ to $y$
only depends on the empirical measure of the particles.

Mean-field systems are commonly used in physics and biology to model
the evolution of a system where the influence of all particles on a
tagged particle is the average of the force exerted by each particle
on the tagged particle.  In the recent work \cite{Budhiraja2014b}, it
was shown that whenever $Q$ satisfies suitable conditions a free
energy of the form
\begin{equation}\label{e:intro:FreeEnergy}
\mathcal{\cF}(\mu)=\sum_{x\in \cX} \mu_{x} \log \mu_{x} + \sum_{x\in \cX} \mu_{x} K_{x}(\mu)
\end{equation}
for some appropriate potential $K:\cP(\cX)\times \cX \to \R$ (see
Definition~\ref{def:GibbsPotential}) is a Lyapunov function for the
evolution equation \eqref{eq_evol_nonlin}, i.e. it decreases along the
flow.

In this work, we show that this monotonicity is actually a consequence
of a more fundamental underlying structure. Namely, we exibit a novel
geometric structure on the space of probability measure $\cP(\cX)$
that allows to view the evolution equation \eqref{eq_evol_nonlin} as
the gradient flow of the free energy $\cF$.

This gradient flow structure is a natural non-linear extension of the
discrete gradient flow structures that were discovered in
\cite{Maas2011} and \cite{Mielke2013} in the context of linear
equations describing Markov chains or more generally in
\cite{Erbar2014} in the context of L\'evy processes.

Moreover, we shall show that our new gradient flow structure for the
non-linear equation arises as the limit of the gradient flow structures
associated to a sequence of mean-field $N$ particle Markov chains. As
an application, we use the stability of gradient flows to show
convergence of these mean-field dynamics to solutions of the
non-linear equation \eqref{eq_evol_nonlin}, see Theorem
\ref{thm:PStoMF}.

\subsection{Gradient flows in spaces of probability measures}

Classically, a gradient flow is an ordinary differential equation of the form
\[
  \dot{x}(t) = -\nabla \cF(x(t)).
\]
By now there exists an extensive theory, initiated by De Giorgi and his
collaborators \cite{DeGiorgi1980}, giving meaning to the notion of
gradient flow when the curve $x$ takes values in a metric space.

Examples of these generalized gradient flows are the famous results by
Otto \cite{Jordan1998,Otto2001} stating that many diffusive partial
differential equations can be interpreted as gradient flows of
appropriate energy functionals on the space of probability measures on
$\R^d$ equipped with the $L^2$ Wasserstein distance. These include the
Fokker--Planck and the pourous medium equations. An extensive
treatment of these examples in the framework of De Giorgi was
accomplished in \cite{Ambrosio2008}.

Gradient flow structures allow to better understand
the role of certain Lyapunov functionals as thermodynamic free
energies. Recently, also unexpected connections of gradient flows to large deviations
have been unveiled \cite{Adams2011} \cite{Dirr2012},
\cite{Duong2013b}, \cite{Erbar2015}, \cite{Fathi2014}.

Since the heat equation is the PDE that governs the evolution of the
Brownian motion, a natural question was whether a similar structure
can be uncovered for reversible Markov chains on discrete spaces. This
question was answered positively in works of Maas \cite{Maas2011} and
Mielke \cite{Mielke2013}, which state that the evolution equations
associated to reversible Markov chains on finite spaces can be
reformulated as gradient flows of the entropy (with respect to the
invariant measure of the chain) for a certain metric structure on the
space of probability measures over the finite space. In
\cite{ErbarMaas2014}, a gradient flow structure for discrete porous
medium equations was also uncovered, based on similar ideas.

In Section 2, we shall highlight a gradient flow structure for
\eqref{eq_evol_nonlin}, which is a natural non-linear
generalization of the structure discovered in \cite{Maas2011} and
\cite{Mielke2013} for such non-linear Markov processes.
This structure explains why the non-linear entropies of
\cite{Budhiraja2014b} are Lyapunov functions for the non-linear ODE.
Moreover, we shall show in Section 3 that this structure is compatible
with those of \cite{Maas2011} and \cite{Mielke2013}, in the sense that
it arises as a limit of gradient flow structures for $N$-particle
systems as $N$ goes to infinity.

\subsection{Convergence of gradient flows}

Gradient flows have proven to be particularly useful for the study of convergence of sequences of evolution
equations to some limit since they provide a very rigid structure.
Informally, the philosophy can be summarized as follows: consider a sequence of
gradient flows, each associated to some energy functional
$\mathcal{F}^{N}$ and some metric structure. If the sequence
$\mathcal{F}^{N}$ converges in some sense to a limit
$\mathcal{F}^{\infty}$ and if the metric structures converge to some
limiting metric, then one would expect the sequence of gradient flows
to converge to a limit that can be described as a gradient flow of
$\mathcal{F}^{\infty}$ for the asymptotic metric structure.

There are several ways of rigorously implementing this philosophy
to actually prove convergence in concrete situations. The one we shall
be using in this work is due to Sandier and Serfaty in
\cite{Sandier2004}, and was later generalized
in~\cite{Serfaty2011}. Other methods, based on discretization schemes,
have been developed in~\cite{AmbrosioSavareZambotti2009}
and~\cite{Daneri2010}. See also the recent survey~\cite{Mielke2014} for an
extension of the theory to generalized gradient systems.
In the context of diffusion equations,
arguments similar to those of~\cite{Serfaty2011} have been used
in~\cite{Fathi2014} to study large deviations.

In the discrete setting, we can combine the framework
of~\cite{Maas2011} and~\cite{Mielke2013} with the method
of~\cite{Serfaty2011} to study scaling limits of Markov chains on
discrete spaces. In this work, we shall use this method to study
scaling limits of $N$-particle mean-field dynamics on finite
spaces. While the convergence statement could be obtained
through more classical techniques, such as those of
\cite{Oelschlager84,Sznitman1989}, our focus here is on justifying
that the gradient flow structure we present is the natural one,
since it arises as the limit of the gradient flow structures for the
$N$-particle systems.

While we were writing this article, we have been told by Maas and
Mielke that they have also successfully used this technique to study
the evolution of concentrations in chemical reactions. We also mention
the work \cite{GigliMaas}, which showed that the metric associated to
the gradient flow structure for the simple random walk on the discrete
torus $\mathbf{Z}/N\mathbf{Z}$ converges to the Wasserstein structure
on $\cP(\mathbf{T})$, establishing compatibility of the discrete and
continuous settings in a typical example. The technique can also be
used to prove convergence of interacting particle systems on lattices,
such as the simple exclusion process (see \cite{FathiSimon2015}). The
technique is not restricted to the evolution of probability measures
by Wasserstein-type gradient flows, but can be also applied for
instance to coagulation-fragmentation processes like the
Becker-D\"oring equations, where one can prove macroscopic limits (see
\cite{Schlichting2016}).

\subsection{Continuous mean-field particle dynamics}

Let us briefly compare the scaling limit for discrete
mean-field dynamics considered in this paper with the more classical
analogous scaling limit for particles in the continuum decribed by
McKean-Vlasov equations.

$N$-particle mean-field dynamics describe the behavior of $N$
particles given by some spatial positions $x_1(t),\dots,x_N(t)$, where
each particle is allowed to interact through the empirical measure of
all other particles.

In nice situations, when the number of particles goes to infinity, the
\emph{empirical measure} of the system $\frac{1}{N}\sum
\delta_{x_i(t)}$ converges to some probability measure $\mu(t)$, whose
evolution is described by a McKean-Vlasov equation. In the continuous
setting, with positions in~$\R^d$, this can be for example a PDE of
the form
\[
  \partial_t \mu(t) = \Delta\mu(t) + \operatorname{div}(\mu(t)(\nabla W *\mu(t)))
\]
where $\nabla W * \mu$ is the convolution of $\mu$ with an interaction
that derives from a potential~$W$. The according free energy in this
case is given by
\[
  \cF(\mu) = \begin{cases}\int \frac{d\mu}{dx}(x)\log \frac{d\mu}{dx}(x)\; dx + \frac{1}{2} \int \int W(x-y)\mu(dx)\mu(dy), & \mu \ll \mathcal{L},\\ \infty ,& \text{otherwise}, \end{cases}
\]
i.e.~formally $K_{x}(\mu) = \frac{1}{2} (W * \mu)(x)$
in~\eqref{e:intro:FreeEnergy}. More general PDEs, involving diffusion
coefficients and confinement potentials, are also possible. We refer
to \cite{Dobrushin1979,Sznitman1989} for more information on
convergence of $N$-particle dynamics to McKean-Vlasov equations. We
also refer to~\cite{Dawson1987, dPdH1996} for the large deviations
behavior. An important consequence of this convergence is that, for
initial conditions for which the particles are exchangeable, there is
\emph{propagation of chaos}: the laws of two different tagged particles
become independent as the number of particles goes to infinity
\cite[Proposition 2.2]{Sznitman1989}.

It has been first noted in \cite{Carrillo2003} that McKean-Vlasov
equations on $\R^d$ can be viewed as gradient flows of the free energy
in the space of probability measures endowed with the Wasserstein
metric. This fact has been useful in the study of the long-time
behavior of these equations
(cf.~\cite{Carrillo2003,Carrillo2006,Cattiaux2008,Malrieu2003} among
others). The study of long-time behavior of particle systems on finite
spaces has attracted recent interest (see for example
\cite{LevinLuczakPeres2010} for the mean-field Ising model), and we
can hope that curvature estimates for such systems may be useful to
tackle this problem, as they have been in the continuous
setting. Since lower bounds on curvature are stable, the study of
curvature bounds for the mean field limit (which is defined as
convexity of the free energy in the metric structure for which the
dynamics is a gradient flow, see for example \cite{EM11}) can shed
light on this problem. We leave this issue for future work. We must
also mention that Wasserstein distances have also been used to
\emph{quantify} the convergence of mean-field systems to their scaling
limit, see for example \cite{BolleyGuillinVillani2007}.

\subsection{Outline}
In Section~\ref{S:GF:MF}, we introduce the gradient flow structure of
the mean-field system~\eqref{eq_evol_nonlin} on discrete spaces.  In
Section~\ref{S:Limit}, we will obtain this gradient flow structure as
the limit of the linear gradient flow structure associated with an
$N$-particle system with mean-field interaction. The nonlinear
gradient flow structure comes with a metric, whose properties will be
studied in Section~\ref{S:metric}.  We close the paper with two
Appendices~\ref{S:stirling} and~\ref{S:variance}, in which auxiliary
results for the passage to the limit are provided.

\section{Gradient flow structure of mean-field systems on discrete spaces}\label{S:GF:MF}

In this section, we derive the gradient flow formulation for the
mean-field system~\eqref{eq_evol_nonlin}. First, we introduce the
metric concept of gradient flows in Section~\ref{sec:metric-gf}. Then
in Section~\ref{S:GF:discrete} we turn to the discrete setting. We
give the precise assumptions on the Markov transition kernel $Q(\mu)$
in the system~\eqref{eq_evol_nonlin} and state several necessary
definitions for later use. In Section~\ref{S:GF:CE}, we define curves
of probability measures via a continuity equation and associate to
them an action measuring their velocity. Based on this, we can
introduce in Section~\ref{S:GF:metric} a transportation distance on
the space of probability measure on $\cX$. The gradient flow
formulation is proven in Section~\ref{S:GF:GF} as curves of maximal
slope. Finally, the gradient structure is lifted to the space of
randomized probability measures in Section~\ref{S:GF:lift}, which is a
preparation for the passage to the limit.

\subsection{Gradient flows in a metric setting}\label{sec:metric-gf}

Let briefly recall the basic notions concerning gradient flows in
metric spaces. For an extensive treatment we refer to
\cite{Ambrosio2008}.

Let $(M,d)$ be a complete metric space. A curve $(a,b)\ni t \mapsto
{u}(t) \in M$ is said to be locally \emph{$p$-absolutely continuous} if there exists $m\in L^p_{\loc}((a,b))$ such that
\begin{equation}\label{e:def:ACcurve}
  \forall a \leq s < t \leq b: \qquad d({u}(s),{u}(t)) \leq \int_{s}^{t} m(r) \; dr .
\end{equation}
We write for short $u\in\AC^p_{\loc}\!\big((a,b),(M,d)\big)$. For any such curve the
metric derivative is defined by
\begin{align*}
  |{u}'(t)|=\lim_{s\rightarrow t}\frac{d({u}(s),{u}(t))}{|s-t|}.
\end{align*}
The limit exists for a.e.~$t\in(a,b)$ and is the smallest $m$ in
\eqref{e:def:ACcurve}, see \cite[Thm.~1.1.2]{Ambrosio2008}.

Now, let $\varPhi:M\to\R$ be lower semicontinuous function. The metric
analogue of the modulus of the gradient of $\varPhi$ is given by the
following definition.

\begin{definition}[Strong upper gradient]
  A function $G:M\rightarrow[0,\infty],$ is a
  strong upper gradient for $\varPhi$ if for every absolutely
  continuous curve $u:(a,b)\rightarrow M,$ the function $G(u)$ is
  Borel and
  \[|\varPhi(u(s))-\varPhi(u(t))|\leq\int_{s}^{t}G(u(r))|u'|(r)dr,
  \hspace{16pt}\forall a<s\leq t < b.\]
\end{definition}
\noindent By Young's inequality, we see that the last inequality implies that
\begin{align*}
    \varPhi(u(s))-\varPhi(u(t)) \leq \frac12\int_{s}^{t}|u'|^2(r)dr + \frac12\int_{s}^{t}G^2(u(r))dr\;,
\end{align*}
for any absolutely continuous curve $u$ provided the $G$ is a strong
upper gradient.

The following definition formalizes what it means for a curve to be a
gradient flow of the function $\varPhi$ in the metric space
$(M,d)$. Shortly, it is a curve that saturates the previous
inequality.
\begin{definition}[Curve of maximal slope]
  A locally absolutely continuous curve $u:(a,b)\to M$ is called a curve of
  maximal slope for $\varPhi$ with respect to its strong upper
  gradient $G$ if for all $a\leq s\leq t\leq b$ we have the energy
  identity
\begin{align}\label{eq:max-slope}
  \varPhi(u(s))-\varPhi(u(t)) = \frac12\int_s^t|u'|^2(r) d r + \frac12 \int_s^t G^2(u(r)) dr\;.
\end{align}
\end{definition}

When $\varPhi$ is bounded below one has a convenient estimate on the
modulus of continuity of a curve of maximal slope $u$. By H\"older's
inequality and \eqref{eq:max-slope} we infer that for all $s< t$ we
have
\begin{align*}
   d(u(s),u(t)) &\leq \int_s^t|u'|(r) dr \leq \sqrt{t-s}\;\left(\int_s^t|u'|^2(r)dr\right)^{\frac12}\\
                &\leq \sqrt{t-s}\; \sqrt{2 \big(\varPhi\big(u(0)\big)-\varPhi_{\min}\big)}\;.
\end{align*}

\subsection{Discrete setting}\label{S:GF:discrete}

Let us now introduce the setting for the discrete McKean--Vlasov
equations that we consider.

In the sequel, we will denote with $\cP(\cX)$, the space of
probability measures on $\cX$, and $\cP^{*}(\cX)$ the set of all
measures that are strictly positive, i.e.
\[
\mu\in\cP^{*}(\cX) \qquad \text{iff} \qquad \forall x\in\cX: \mu_{x}>0,
\]
and finally with $\cP^{a}(\cX),$ the set of all measures that have
everywhere mass bigger than~$a$, i.e.
\[
\mu\in\cP^{a}(\cX) \qquad \text{iff} \qquad \forall x\in\cX: \mu_{x}\geq a.
\]
As in ~\cite{Budhiraja2014a,Budhiraja2014b}, we shall consider
equations of the form \eqref{eq_evol_nonlin} where $Q$ is Gibbs with
some potential function $K$. Here is the definition of such transition
rates, taken from \cite{Budhiraja2014b}:
\begin{definition}\label{def:GibbsPotential}
  Let $K:\mathcal{P}(\mathcal{X})\times\mathcal{X}\rightarrow\R$ be
  such that for each $x\in \mathcal{X},
  K_{x}:\mathcal{P}(\mathcal{X})\rightarrow\R$ is a twice continuously
  differentiable function on $\mathcal{P}(\mathcal{X})$. A family of
  matrices~$\{Q(\mu)\in \R^{\cX\times \cX}\}_{\mu\in\cP(\cX)}$ is
  Gibbs with potential function $K$, if for each $\mu\in \cP(\cX)$,
  $Q(\mu)$ is the rate matrix of an irreducible, reversible ergodic
  Markov chain with respect to the probability measure
  \begin{equation}\label{e:def:piH}
    \pi_{x}(\mu) = \frac{1}{Z(\mu)}\exp(-H_{x}(\mu))\;,
  \end{equation}
  with 
  \begin{equation}\nonumber
      H_{x}(\mu) = \frac{\partial}{\partial \mu_x}U(\mu)\;,\hspace{4pt}\hspace{4pt} U(\mu)=\underset{x \in \cX}{\sum} \; \mu_x K_x(\mu)\;.
  \end{equation}
  In particular $Q(\mu)$ satisfies the detailed balance condition
  w.r.t.~$\pi(\mu)$, that is for all~$x,y\in \cX$
  \begin{equation}\label{e:DBC}
    \pi_{x}(\mu) Q_{xy}(\mu) = \pi_{y}(\mu) Q_{yx}(\mu)
  \end{equation}
  holds. Moreover, we assume that for each $x,y\in \cX$ the map
  $\mu\mapsto Q_{xy}(\mu)$ is Lipschitz continuous over~$\cP(\cX)$.
\end{definition}
In the above definition and in the following we use the convention that a function $F(\cdot): \cP(\cX)
\to \R$ is regular if and only if it can be extended
in an open neighborhood of $\cP(\cX) \subset \R^{d}$ in which it is
regular in the usual sense. Hereby, regular could be continuous, Lipschitz or differentiable. In particular, we use this for the (twice) continuously differentiable function $K_x$ and the Lipschitz continuous function $Q_{xy}$ from above.
\begin{remark}\label{rem:MarkovKernel}
  There are many ways of building a Markov kernel that is reversible
  with respect to a given probability measure. The most widely used
  method is the Metropolis algorithm, first introduced in
  \cite{MRRTT53}:
  \begin{equation*}
    Q_{xy}^{\operatorname{MH}}(\mu) := \min\left(\frac{\pi_{y}(\mu)}{\pi_{x}(\mu)},1 \right) = e^{-\bra{H_{y}(\mu) - H_{x}(\mu)}_+} , \quad\text{with}\quad (a)_+ := \max\set{0, a} .
  \end{equation*}
  By this choice of the rates it is only necessary
  to calculate $H(\mu)$ and not the partition sum~$Z(\mu)$
  in~\eqref{e:def:piH}, which often is a costly computational
  problem.

  A general scheme for obtaining rates satisfying the detailed balance
  condition with respect to $\pi$ in~\eqref{e:def:piH} is to consider
  \begin{equation*}
    Q_{xy}(\mu) = \frac{\sqrt{\pi_{y}(\mu)}}{\sqrt{\pi_{x}(\mu)}} \ A_{xy}(\mu) ,
  \end{equation*}
  where $\set{A(\mu)}_{\mu\in \cP(\cX)}$ is a family of irreducible symmetric matrices. If we choose
  $A_{xy}(\mu) = \alpha_{x,y} \min
  \left(\frac{\sqrt{\pi_{y}(\mu)}}{\sqrt{\pi_{x}(\mu)}},\frac{\sqrt{\pi_{x}(\mu)}}{\sqrt{\pi_{y}(\mu)}}\right)$
  with $\alpha \in \set{0,1}^{\cX\times \cX}$ an irreducible symmetric adjacency matrix, we recover
  the Metropolis algorithm on the corresponding graph.
\end{remark}

We will be interested in the non-linear evolution equation
\begin{equation}\label{e:main-eq}
\dot{c}_x(t) = \sum_{y\in\cX}{c}_y(t)Q_{yx}({c}(t))\; ,
\end{equation}
with the convention $Q_{xx}(\mu) = -\sum_{y\ne x} Q_{xy}(\mu)$.
By the Lipschitz assumption on $Q$ this equation has a unique solution.

One goal will be to express this evolution as the gradient flow of the
associated free energy functional $\cF:\cP(\cX)\to\R$ defined by
\begin{equation}\label{e:def:MF:FreeEnergy}
  \cF(\mu):=\sum_{x\in\mathcal{X}}\mu_{x}\log\mu_{x} + U(\mu),
  \qquad\text{with}\qquad
  U(\mu):=\sum_{x\in\mathcal{X}}\mu_{x}K_{x}(\mu)\;.
\end{equation}

To this end, it will be convenient to introduce the so-called
\emph{Onsager operator} $\cK(\mu):\R^\cX\to\R^\cX$.
It is defined as follows:

Let $\Lambda:\R_{+}\times\R_{+}\rightarrow
\R_{+},$ denote the \emph{logarithmic mean} given by
\begin{align*}
  \Lambda(s,t):=\int_0^1s^\alpha t^{1-\alpha} \,d\alpha = \frac{s-t}{\log s - \log t} \; .
\end{align*}

$\Lambda$ is continuous, increasing in both variables, jointly concave
and $1$-homogeneous. See for example \cite{Maas2011, EM11} for more
about properties of this logarithmic mean. In the sequel we are going
to use the following notation
\begin{equation}\label{e:MF:weights}
w_{xy}(\mu) := \Lambda(\mu_{x}Q_{xy}(\mu),\mu_{y}Q_{yx}(\mu))
\end{equation}
since this term will appear very often. By the definition of $\Lambda$
and the properties of $Q$ we get that $w_{xy}$ is uniformly bounded on
$\cP(\cX),$ by a constant $C_{w}$.

Now, we can define
\begin{equation}\label{e:def:OnsagerOperator}
  \cK(\mu) := \frac{1}{2}\sum_{x,y}w_{xy}(\mu)\; \bra{e_x - e_y} \otimes \bra{e_x - e_y} ,
\end{equation}
where $\set{e_x}_{x\in \cX}$ is identified with the standard basis of
$\R^\cX$. More explicitly, we have for $\psi\in\R^\cX$:
\begin{align*}
  \big(\cK(\mu)\psi\big)_x = \sum_yw_{xy}(\mu)\big(\psi_x-\psi_y\big)\;.
\end{align*}

With this in mind, we can formally rewrite the evolution \eqref{e:main-eq}
in gradient flow form:
\begin{equation}\label{e:GF:KDF}
  \dot{{c}}(t) = - \cK({c}(t)) D\cF({c}(t)),
\end{equation}
where $D\cF(\mu)\in\R^\cX$ is the differential of $\cF$ given by
$D\cF(\mu)_x=\partial_{\mu_x}\cF(\mu)$.

Finally, let us introduce the \emph{Fisher information}
$\cI:\cP(\cX)\rightarrow[0,\infty]$ defined for $\mu\in\cP^*(\cX)$ by
  \begin{equation}\label{fisher}
    \mathcal{I}(\mu) := \frac{1}{2} \sum\limits_{(x,y)\in E_{\mu}} w_{xy}(\mu)\left(\log(\mu_{x}Q_{xy}(\mu))-\log(\mu_{y}Q_{yx}(\mu))\right)^2
  \end{equation}
where, for $\mu\in\cP(\cX)$, we define the edges of possible jumps by
\begin{equation}\label{e:def:allowsTrans}
    E_\mu := \set{ (x,y)\in \cX\times \cX : Q_{xy}(\mu) > 0 } .
\end{equation}
For $\mu\in\cP(\cX)\setminus\cP^*(\cX)$ we set $\cI(\mu)=+\infty$.

$\cI$ gives the dissipation of $\cF$ along the evolution, namely, if
$c$ is a solution to \eqref{e:main-eq} then
\begin{align*}
  \frac{d}{dt}\cF(c(t)) = - \cI(c(t))\;.
\end{align*}

\subsection{Continuity equation and action}\label{S:GF:CE}

In the sequel we shall use the notation for the discrete
gradient. Given a function $\psi\in\R^{\cX}$ we define
$\nabla\psi\in\R^{\cX\times\cX}$ via $\nabla_{xy}\psi:=\psi_y-\psi_x$
for $x,y\in\cX$. We shall also use a notion of discrete divergence,
given $v\in\R^{\cX\times\cX}$, we define $\delta v\in\R^\cX$ via
$(\delta v)_x=\frac12\sum_y(v_{xy}-v_{yx})$.

\begin{definition}[Continuity equation]\label{def:MF:continuity_equ}
  Let $T>0$ and $\mu,\nu\in \cP(\cX)$. A pair $({c},v)$ is called a
  solution to the continuity equation, for short $(c,v)\in
  \vecfield\CE_T(\mu,\nu)$, if
  \begin{enumerate}[ (i) ]
  \item ${c} \in C^0([0,T],\cP(\cX))$, i.e.\ $\forall x\in \cX:$ $t\mapsto c_x(t) \in C^0([0,T],[0,1])$;
  \item ${c}(0) = \mu ;\, {c}(T)=\nu$;
  \item $v:[0,T]\rightarrow \R^{\cX\times\cX}$ is
    measurable and integrable;
  \item The pair $({c}, v)$ satisfies the continuity equation
    for $t\in (0,T)$ in the weak form, i.e.~for all $\varphi\in
    C^1_c((0,T),\R)$ and all $x\in \cX$ holds
    \begin{equation} \label{e:MF:continuity_equ:weak}\begin{split}
        &\int_0^T\Big[ \dot\varphi(t)\; {c}_{x}(t)- \varphi(t)\;(\delta v)_x(t)\Big]dt = 0 .
      \end{split}\end{equation}
  \end{enumerate}
  In a similar way, we shall write $({c},\psi)\in \CE_T(\mu,\nu)$ if
  $(c,w(c)\nabla\psi)\in \vecfield\CE_T(\mu,\nu)$ for $\psi:
  [0,T]\to \R^\cX$ and $(w(c)\nabla \psi)_{xy}(t) := w_{xy}(c(t))
  \nabla_{xy}\psi(t)$ defined pointwise. In the case $T=1$ we will
  often neglect the time index in the notation setting
  $\vecfield\CE(\mu,\nu):= \vecfield\CE_1(\mu,\nu)$. Also, the
  endpoints $(\mu,\nu)$ will often be suppressed in the notation.
\end{definition}
To define the action of a curve it will be convenient to introduce the
function $\alpha:\R\times\R_{+}\to\R_{+}$ defined by
\begin{equation}\label{e:def:alpha}
  \alpha(v,w) := \begin{cases}
    \frac{v^2}{w} &, w>0 \\
    0 &, v=0=w \\
    +\infty &, \text{else }
  \end{cases} \ \ .
\end{equation}
Note that $\alpha$ is convex and lower semicontinuous.
\begin{definition}[Curves of finite action]
  Given $\mu\in\cP(\cX)$, $v\in\R^{\cX\times\cX}$ and $\psi\in\R^\cX$, we define the
  \emph{action} of $(\mu,v)$ and $(\mu,\psi)$ via
  \begin{align}\label{e:def:MF:action:vec}
    \vecfield\cA(\mu,v) &:=\frac{1}{2}\sum_{x,y}  \alpha\!\bra{v_{xy},w_{xy}(\mu)}\;,\\
    \cA(\mu,\psi) &:= \vecfield \cA(\mu,w(\mu)\nabla\psi) = \frac{1}{2} \sum_{x,y} \bra{\psi_y - \psi_x}^2 w_{xy}(\mu)\;.\notag
  \end{align}
  Moreover, a solution to the continuity equation $(c,v)\in \CE_T$
  is called \emph{a curve of finite action} if
  \begin{equation*}
    \int_0^T \vecfield\cA({c}(t),v(t)) \; dt < \infty\;.
  \end{equation*}
\end{definition}
It will be convenient to note that for a given solution $(c,v)$ to the
continuity equation we can find a vector field $\tilde v=w(c)\nabla\psi$ of
gradient form such that $(c,\tilde v)$ still solves the continuity equation
and has lower action.
\begin{proposition}[Gradient fields]\label{prop:Gradient fields}
  Let $({c},v)\in \vecfield\CE_T(\mu,\nu)$ be a curve of finite
  action, then there exists $\psi : [0,T]\to \mathbb{R}^{\cX}$
  measurable such that $({c},\psi)\in \CE_T(\mu,\nu)$ and
  \begin{equation}\label{e:GradientFields:Ainequ}
    \int_0^T \cA({c}(t),\psi(t))  \; dt \leq \int_0^T \vecfield\cA({c}(t),v(t))  \; dt.
  \end{equation}
\end{proposition}
\begin{proof}
  Given ${c}\in\cP(\cX)$ we will endow $\R^{\cX\times\cX}$ with the
  weighted inner product
  \begin{align*}
    \langle\Psi,\Phi\rangle_{\mu}:=\frac12\sum_{x,y}\Psi_{xy}\Phi_{xy}w_{xy}({\mu})\;,
  \end{align*}
  such that $\vec\cA({\mu},v)=|\Psi|_{\mu}^2$ if $v_{xy}=w_{xy}\Psi_{xy}$. Denote by
  $\mathrm{Ran}(\nabla):=\{\nabla\psi:\psi\in\R^\cX\}\subset\R^{\cX\times\cX}$
  the space of gradient fields. Moreover, denote by
  $$\mathrm{Ker}(\nabla^*_{\mu}):=\left\{\Psi\in\R^{\cX\times\cX}:\sum_{x,y}(\Psi_{yx}-\Psi_{xy})w_{xy}({\mu})=0\right\}$$
  the space of divergence free vector fields. Note that we have the
  orthogonal decomposition
  \begin{align*}
    \R^{\cX\times\cX}=\mathrm{Ker}(\nabla^*_{\mu})\oplus^\perp \mathrm{Ran}(\nabla)\;.
  \end{align*}
  Now, given $({c},v)\in \vecfield\CE_T(\mu,\nu)$, we have
  $\vec\cA(c(t),v(t))<\infty$ for a.e.~$t\in [0,T]$. Thus, from
  ~\eqref{e:def:alpha} we see that for a.e.~$t$ and all $x,y$ we have
  that $v_{xy}(t)=0$ whenever $w_{xy}(c(t))=0$. Hence, we can define
  \begin{align*}
    \Psi_{xy}(t) := \frac{v_{xy}(t)}{w_{xy}(c(t))}\qquad\text{for a.e.~$t\in[0,T]$.}
  \end{align*}
  Then $\psi:[0,T]\to\R^\cX$ can be given by setting $\nabla\psi(t)$ to be the
  orthogonal projection of $\Psi(t)$ onto $\mathrm{Ran}(\nabla)$
  w.r.t.~$\langle\cdot,\cdot\rangle_{{c}(t)}$. The orthogonal
  decomposition above then implies immediately that
  $({c},w\nabla\psi)\in\vecfield\CE_T(\mu,\nu)$ and that
  $|\nabla\psi(t)|^2_{{c}(t)}\leq |\Psi(t)|^2_{{c}(t)} = \vec \cA(c(t),v(t))$ for a.e.~$t\in[0,T]$. This yields~\eqref{e:GradientFields:Ainequ}.
\end{proof}
\subsection{Metric}\label{S:GF:metric}
We shall now introduce a new transportation distance on the space
$\cP(\cX)$, which will provide the underlying geometry for the gradient
flow interpretation of the mean field evolution equation
\eqref{eq_evol_nonlin}.
\begin{definition}[Transportation distance]\label{def:def:W}
  Given $\mu,\nu\in\cP(\cX)$, we define
  \begin{equation}\label{e:def:W}
    \cW^{2}(\mu, \nu) := \inf \set{ \int_0^1 \cA({c}(t), \psi(t)) \; dt : ({c},\psi)\in  \CE_{1}(\mu,\nu)} .
  \end{equation}
\end{definition}

\begin{remark}\label{rem:equiv-formulation}
  As a consequence of Proposition \ref{prop:Gradient fields} and the
  fact that for any $\mu\in\cP(\cX)$ and $\psi\in\R^\cX$ give rise to
  $v\in\R^{\cX\times\cX}$ via $v_{xy}=w_{xy}(\mu)\nabla_{xy}\psi$ such
  that $\cA(\mu,\psi)=\vec\cA(\mu,v)$ we obtain an equivalent
  reformulation of the function $\cW$:
  \begin{equation*}
    \cW^{2}(\mu, \nu) = \inf \set{ \int_0^1 \vec\cA({c}(t), v(t)) \; dt : ({c},v)\in  \vec\CE_{1}(\mu,\nu)} .
  \end{equation*}
\end{remark}
It turns out that $\cW$ is indeed a distance.
\begin{proposition} \label{prop_metric} The function $\cW$ defined in Definition
  \ref{def:def:W} is a metric and the metric space $(\cP(\cX),\cW)$
  is seperable and complete. Moreover, any two points
  $\mu,\nu\in\cP(\cX)$ can be joined by a constant speed
  $\cW$-geodesic, i.e.~there exists a curve $(\gamma_t)_{t\in[0,1]}$
  with $\gamma_0=\mu$ and $\gamma_1=\nu$ satisfying
  $\cW(\gamma_s,\gamma_t)=|t-s|\cW(\mu,\nu)$ for all $s,t\in[0,1]$.
\end{proposition}
We defer the proof of this statement until Section~\ref{S:metric}. Let us give a
characterization of absolutely continuous curves w.r.t.~$\cW$.
\begin{proposition}\label{accurves}
  A curve ${c}:[0,T]\rightarrow \mathcal{P}(\cX)$ is absolutely
  continuous w.r.t.~$\cW$ if and only if there exists
  $\psi:[0,T]\times\cX\to\R$ such that $(c,\psi)\in\CE_T,$ and
  $\int_0^T\!\!\sqrt{ \cA({c}(t), \psi(t))} \, dt<\infty$. Moreover, we
  can choose $\psi$ such that the metric derivative of $c$ is given as
  $|{c}'(t)|=\sqrt{\mathcal{A}({c}(t),\psi(t))}$ for a.e.~$t$.
\end{proposition}
\begin{proof}
  The proof is identical to the one of
  \cite[Thm.~5.17]{Dolbeault2008}.
\end{proof}
\subsection{Gradient flows}\label{S:GF:GF}
In this section, we shall present the interpretation of the discrete
McKean-Vlasov equation as a gradient flow with respect to the distance~$\cW$.
We will use the abstract framework introduced in Section
\ref{sec:metric-gf} above, where $(M,d)=(\cP(\cX),\cW)$ and
$\varPhi=\cF$.
\begin{lemma}
  Let $\mathcal{I}:\cP(\cX)\rightarrow[0,\infty]$ defined in
  \eqref{fisher} denote the Fisher information and let
  $\cF:\cP(\cX)\to \R$ defined in \eqref{e:def:MF:FreeEnergy} denote
  the free energy. Then, $\sqrt{\cI}$ is a strong upper gradient for
  $\cF$ on $(\cP(\cX),\cW)$, i.e. for $(c,\psi)\in \CE_T$
  and $0\leq t_1 < t_2 \leq T$ holds
  \begin{equation}\label{e:strongupper}
    \abs{\cF(c(t_2)) - \cF(c(t_1))}
    \leq \int_{t_1}^{t_2} \sqrt{\cI(c(t))} \sqrt{\cA(c(t),\psi(t))} \; dt .
  \end{equation}
\end{lemma}
\begin{proof}
  Let ${c}:(a,b)\rightarrow(\cP(\cX),\cW)$ be a $\cW-$absolutely
  continuous curve with $\psi$ the associated gradient potential such
  that $({c},\psi)\in \CE(c(a),c(b))$ and
  $|{c}'|(t)=\sqrt{\cA({c}(t),\psi(t))}$ for a.e.~$t\in (a,b)$. We can
  assume w.l.o.g.\ that the r.h.s.\ of~\eqref{e:strongupper} is finite.
  For the proof, we are going to define the auxiliary
  functions \[\cF_{\delta}(\mu)=\sum_{x\in\cX}(\mu_{x}+\delta)\log(\mu_{x}+\delta)
  + U(\mu).\] The function $\cF_{\delta}(\mu)$ are Lipschitz
  continuous and converge uniformly to $\cF$, as $\delta\rightarrow
  0$. By Lemma \ref{lemma:bounds}, ${c}$ is also absolutely
  continuous with respect to Euclidean distance. Therewith, since
  $\cF_{\delta}$ are Lipschitz continuous, we have that
  $\cF_{\delta}({c}):(a,b)\rightarrow\mathbb{R}$ is absolutely
  continuous and hence
  \[\cF_{\delta}({c}(t_{2}))-\cF_{\delta}({c}(t_{1}))=\int_{t_{1}}^{t_{2}}
  \frac{d}{dt}
  \cF_{\delta}({c})(t)dt=\int_{t_{1}}^{t_{2}}D\cF_{\delta}({c}(t))\dot{{c}}(t)\;
  dt ,\] where $D\cF_{\delta}({c}(t))$ is well-defined for a.e.~$t$
  and given in terms of
  \[\partial_{e_x} \cF_{\delta}({c}(t)) = H_{x}({c}(t)) + \log
  ({c}_{x}(t)+\delta) + 1.\] Now, we have by using the Cauchy-Schwarz
  inequality
  \begin{align*}
    &\int_{t_{1}}^{t_{2}}|D\cF_{\delta}({c}(t))\dot{{c}}(t)|\; dt\leq\int_{t_{1}}^{t_{2}}\bigg|\frac{1}{2}\sum_{x,y\in\cX}(\psi_{x}-\psi_{y})(\partial_{e_x} \cF_{\delta}({c})-\partial_{e_y} \cF_{\delta}({c}))w_{xy}({c})\bigg|\; dt\\
    &\leq \int_{t_{1}}^{t_{2}}\sqrt{\frac{1}{2}\sum_{x,y\in\cX}\big(\nabla_{xy}\psi\big)^{2}w_{xy}({c})}  \sqrt{\frac{1}{2}\sum_{x,y\in\cX}\big(\partial_{e_x} \cF_{\delta}({c})-\partial_{e_y} \cF_{\delta}({c})\big)^{2}w_{xy}({c})}\; dt\\
    &= \int_{t_{1}}^{t_{2}}\sqrt{\cA({c},\psi)} \ \ \times\\
    &\qquad\qquad\sqrt{\frac{1}{2}\sum_{x,y\in\cX} \big( \log( {c}_{x}+\delta)  + H_{x}({c}) -\log ({c}_{y}+\delta)    -H_{y}({c})\big)^{2}w_{xy}({c})}\; dt. \\
    &\leq
    \int_{t_{1}}^{t_{2}}\sqrt{\cA({c},\psi)}\sqrt{2\cI({c})+C^2_H \sum_{x,y\in\cX}
      w_{xy}(c) }\; dt ,
  \end{align*}
  where we dropped the $t$-dependence on $c$ and $\psi$.
  For the last inequality, we observe that since $H_{x}$ for $x
  \in\cX$ are uniformly bounded and for $a<b, \delta>0$ it holds
  $\frac{b}{a}\geq\frac{b+\delta}{a+\delta}$, it is easy to see that
  the quantity $ |\log ({c}_{x}(t)+\delta) + H_{x}({c}(t))
  -\log({c}_{y}(t)+\delta) -H_{y}({c}(t))|$ is bounded by $|\log
  ({c}_{x}(t)) + H_{x}({c}(t)) -\log({c}_{y}(t)) -H_{y}({c}(t))|+C_H$
  with $C_H$ only depending on $H$. Moreover, we observe that by
  definitions of $H_x$ and $\pi$ from~\eqref{e:def:piH}, it holds
  \begin{align*}
    &\abs{\log ({\mu}_{x})  + H_{x}({\mu}) -\log({\mu}_{y}) -H_{y}({\mu})} = \abs{\log\frac{\mu_x}{\pi_x(\mu)}- \log\frac{\mu_y}{\pi_y(\mu)}} \\
    &= \abs{\log\bra{\mu_x Q_{xy}(\mu)} - \log\bra{\pi_x(\mu) Q_{xy}(\mu)} -
      \log\bra{\pi_y(\mu) Q_{yx}(\mu)} - \log\bra{\mu_y Q_{yx}(\mu)}} .
  \end{align*}
  Then, by the detailed balance condition~\eqref{e:DBC} the two middle
  terms cancel and we arrive at $\cI(\mu)$. Since, we assumed
  $\int_{t_{1}}^{t_{2}}\sqrt{\cA({c}(t),\psi(t))}\sqrt{\cI({c}(t))}\;
  dt$ to be finite, we can apply the dominated convergence theorem and
  get the conclusion.
\end{proof}
\begin{proposition}
  \label{prop_max_slope_nonlin}
  For any absolutely continuous curve $({c}(t))_{t \in [0,T]}$ in
  $\mathcal{P}(\cX)$ holds
  \begin{align}\label{e:LDfunctional}
    \cJ({c}) := \cF({c}(T)) - &\cF({c}(0)) +
    \frac{1}{2}\int_0^T{\mathcal{I}({c}(t))\; dt} +
    \frac{1}{2}\int_0^T{\mathcal{A}({c}(t), \psi(t))\; dt} \geq 0
  \end{align}
  Moreover, equality is attained if and only if $({c}(t))_{t \in
  [0,T]}$ is a solution to \eqref{eq_evol_nonlin}. In this case
  ${c}(t)\in\cP^{*}(\cX)$ for all $t>0$.
\end{proposition}
In other words, solution to \eqref{eq_evol_nonlin} are the only
gradient flow curves (i.e.~curves of maximal slope) of $\cF$.
\begin{proof}
  The first statement follows as above by Young's inequality from the
  fact that $\cI$ is strong upper gradient for $\cF$.

  Now let us assume that for a curve ${c},$ $\cJ({c})\leq0$ holds. Then
  since \eqref{e:strongupper} holds for every curve we can deduce that we
  actually have
  \[
    \cF({c}(t_{2}))-\cF({c}(t_{1}))+\frac{1}{2}\int_{t_{1}}^{t_{2}}\cI({c}(t))dt+\frac{1}{2}\int_{t_{1}}^{t_{2}}\cA({c}(t),\psi(t))dt=
  0,\hspace{8pt} 0\leq t_{1}\leq t_{2}\leq T.
  \]
  Since $\int_{0}^{T}\cI({c}(t))dt<\infty$, we can find a sequence
  $\epsilon_{n}$, converging to zero, such that
  $\cI({c}(\epsilon_{n}))<\infty.$ By continuity of ${c},$ we can find
  $a,\epsilon>0$, such that ${c}(t)\in\cP^{a}(\cX)$, for
  $t\in[\epsilon_{n},\epsilon_{n}+\epsilon]$. Now, since $\cI$ is
  Lipschitz in $\cP^{a}(\cX),$ we can apply the chain rule for
  $\epsilon_{n}\leq t_{1}\leq t_{2}\leq \epsilon_n+\epsilon$ and get
  \begin{align*}
    \cF({c}(t_{1}))-\cF({c}(t_{2}))&=\int_{t_{1}}^{t_{2}} \skp{D\cF({c}(t))}{\cK({c}(t)) \nabla\psi(t)}\\
    &=\frac{1}{2}\int_{t_{1}}^{t_{2}}\cA({c}(t),\psi(t))dt+\frac{1}{2}\int_{t_{1}}^{t_{2}}\cI({c}(t))dt,
  \end{align*}
  by comparison we get
  \[
    \skp{D\cF({c}(t))}{\cK({c}(t))
    \nabla\psi(t)}=\sqrt{\cA({c}(t),\psi(t))\;
    \cI({c}(t))}=\cA({c}(t),\psi(t))=\cI({c}(t)),
  \]
  for $t\in[\epsilon_{n},\epsilon_{n}+\epsilon]$.
  From which, by an   application of the inverse of the Cauchy-Schwarz inequality,
  we get that
  $\psi_{x}(t)-\psi_{y}(t)=\partial_{e_{x}}\cF({c}(t))-\partial_{e_{y}}\cF({c}(t))$.
  Now we have
  \begin{equation}
    \label{cK}
    \begin{split}
      \dot{{c}}(t)&=-\cK({c}(t)) D\cF({c}(t))\\
      &= -\frac{1}{2} \sum_{x,y} w_{xy}({c}(t))\; \bra{e_x - e_y} \bra{\partial_{e_x} \cF({c}(t)) - \partial_{e_y} \cF({c}(t))} \\
      &= -\frac{1}{2} \sum_{x,y} \bra{Q_{xy}({c}(t)) {c}_{x}(t) - Q_{yx}(c(t)) {c}_{y}(t)} \bra{e_x - e_y} \\
      &= -\sum_x \bra{ \sum_y \bra{Q_{xy}({c}(t)) {c}_{x}(t) -
          Q_{yx}(c(t)) {c}_{y}(t)} }e_x = {c}(t) Q({c}(t))
    \end{split}\end{equation}
  on the interval $[\epsilon_{n},\epsilon_{n}+\epsilon]$. We actually have that
  ${c}(t)$ is a solution to $ \dot{{c}}(t)= {c}(t) Q({c}(t))$ on
  $[\epsilon_{n},T].$ Indeed, let $T_{n}=\sup\{t'\leq T: \dot{{c}}(t)=
  {c}(t) Q({c}(t)), \forall t\in[\epsilon_{n},t']\}.$ We have
  ${c}(T_{n})\in\cP^{b}(\cX),$ for some $b>0,$ because ${c}$ is a
  solution to $\dot{{c}}(t)= {c}(t) Q({c}(t)),$ on
  $[\epsilon_{n},T_{n})$ and the dynamics are irreducible. Now if
  we apply the same argument for $T_{n},$ that we used for
  $\epsilon_{n},$ we can extent the solution beyond $T_{n}.$ If
  $T_{n}<T,$ then we will get a contradiction, Therefore $T_{n}=T.$
  Now by sending $\epsilon_{n}$ to zero we get that ${c}$ is a
  solution to $\dot{{c}}(t)= {c}(t) Q({c}(t)),$ on $[0,T]$.

  Now on the other hand if ${c}$ is a solution to $\dot{{c}}(t)=
  {c}(t) Q({c}(t))$ on $[0,T],$ we can get that for every
  $\epsilon>0,$ there exists $a>0,$ such that ${c}(t)\in\cP^{a}(\cX)$ on
  $[\epsilon,T].$ The choice $\psi(t)=DF(t),$ satisfies the continuity
  equation (see \eqref{cK}), and by applying the chain rule, we get
  that
  \[\cF({c}(T))-\cF({c}(\epsilon))+\frac{1}{2}\int_{\epsilon}^{T}\cI({c}(t))dt+\frac{1}{2}\int_{\epsilon}^{T}\cA({c}(t),\psi(t))dt=
  0.\] Sending $\epsilon$ to zero concludes the proof.
\end{proof}
\begin{remark}
  Note that the formulation above contains the usual entropy
  entropy-production relation for gradient flows. If ${c}$ is a
  solution to \eqref{eq_evol_nonlin}, then $\psi(t) = -D\cF({c}(t))$ and especially it holds
  that $\cA\bra{{c}(t), - D\cF({c}(t))} = \cI({c}(t))$. Therewith,
  \eqref{e:LDfunctional} becomes
  \begin{equation*}
    \cF({c}(T)) + \int_0^T{\mathcal{I}({c}(t))\; dt} = \cF(c(0)) .
  \end{equation*}
\end{remark}
\subsection{Lifted dynamics on the space of random measures}\label{S:GF:lift}
It is possible to lift the evolution $\dot {c}(t) = {c}(t) Q({c}(t))$
in $\cP(\cX)$ to an evolution for measures $\CC$ on $\cP(\cX)$. This
is convenient, if one does not want to start from a deterministic
point but consider random initial data. The evolution is then formally
given by
\begin{equation}\label{e:Liouville}
  \partial_t \mathbb{{C}} (t,\nu) + \Div_{\cP(\cX)}\bra{ \mathbb{{C}}(t,\nu) \bra{\nu Q(\nu)}} = 0 , \quad\text{with}\quad \Div_{\cP(\cX)} = \sum_{x\in \cX} \partial_{e_x}  .
\end{equation}
\begin{notation}
  In the following, all quantities connected to the space
  $\cP(\cP(\cX))$ will be denoted by blackboard-bold letters, like for
  instance random probability measures $\MM \in \cP(\cP(\cX))$ or
  functionals $\FF : \cP(\cP(\cX)) \to \R$.
\end{notation}
The evolution \eqref{e:Liouville} also has a natural gradient flow structure that is
obtained by lifting the gradient flow structure of the underlying
dynamics. In fact, \eqref{e:Liouville} will turn out to be a gradient
flow w.r.t.~to the classical $L^2$-Wasserstein distance on
$\cP(\cP(\cX))$, which is build from the distance $\cW$ on the base
space $\cP(\cX)$. To establish this gradient structure, we need to
introduce lifted analogues of the continuity equation and the action
of a curve as well as a probabilistic representation result for the
continuity equation.
\begin{definition}[Lifted continuity equation]\label{def:Lio:CE}
  A pair $(\mathbb{{C}},\VV)$ is called a solution to the lifted
  continuity equation, for short $(\CC,\VV)\in \vec\CCEE_T(\MM,\NN)$, if
 \begin{enumerate}[ (i) ]
 \item $[0,T]\ni t \mapsto \mathbb{{C}}(t) \in \cP(\cP(\cX))$ is
   weakly$^*$ continuous,
 \item $\mathbb{{C}}(0)= \MM ;\, \mathbb{{C}}(T) = \NN$;
 \item $\mathbb{\VV}:[0,T]\times \cP(\cX)\rightarrow
   \R^{\cX\times\cX}$ is measurable and integrable w.r.t.~$\CC(t,d\mu)dt$,
 \item The pair $(\mathbb{{C}}, \VV)$ satisfies the
   continuity equation for $t\in (0,T)$ in the weak form, i.e.~for all
   $\varphi\in C^1_c((0,T)\times \cP(\cX))$ holds
   \begin{equation} \label{e:Lio:continuity_equ:weak}\begin{split}
       &\int_0^T \int_{\cP(\cX)} \bra{ \dot\varphi(t,\nu)-
         \skp{\nabla\varphi(t,\nu)}{\delta\VV(t,\nu)} }
       \mathbb{{C}}(t,d\nu) \;dt = 0 ,
     \end{split}
   \end{equation}
 \end{enumerate}
where $\delta\VV: \cP(\cX)\rightarrow
\R^{\cX\times\cX},$ is given by $\delta\VV(\nu)_x := \frac{1}{2} \sum_{y} (\VV_{xy}(\nu)-\VV_{yx}(\nu)).$
\end{definition}
Here we consider $\cP(\cX)$ as a subset of Euclidean space
$\R^\cX$ with $\skp{\cdot}{\cdot}$ the usual inner product. In particular,
$\nabla\varphi(t,\mu)=(\partial_{\mu_x}\varphi(t,\mu))_{x\in\cX}$
denotes the usual gradient on $\R^\cX$ and we have explicitly
\begin{align*}
  \skp{\nabla\varphi(t,\nu)}{\delta\VV(t,\nu)} &= \sum_{x\in\cX} \partial_{\mu_x}\varphi(t,\mu)(\delta \VV(t,\mu))_x \\
  &= \frac12\sum_{x,y\in\cX} \partial_{\mu_x}\varphi(t,\mu)\Big(\VV_{xy}(t,\mu)-\VV_{yx}(t,\mu)\Big).
\end{align*}
Thus, \eqref{e:Lio:continuity_equ:weak} is simply the weak formulation
of the classical continuity equation in $\R^\cX$.
In a similar way, we shall write $(\CC,\PPsi)\in \CCEE_T(\MM,\NN)$ if
$\PPsi:[0,T]\times\cP(\cX)\to\R^\cX$ is a function such that
$(\CC,\tilde\VV)\in\vec\CCEE_T(\MM,\NN)$ with
$\tilde\VV_{xy}(t,\mu)=w_{xy}(\mu)\nabla_{xy}\PPsi(t,\mu)$. In this
case we have that $\delta \VV(\mu)=\cK(\mu)\PPsi(\mu)$, where
$\cK(\mu)$ is the Onsager operator defined in
\eqref{e:def:OnsagerOperator}.
Solutions to~\eqref{e:Liouville} are understood as weak solutions like
in Definition~\eqref{def:Lio:CE}. That is $\CC$ is a weak solution to~\eqref{e:Liouville} if $(\CC,\PPsi^*) \in \CCEE_T$ with $\PPsi^*(\nu):= - D\cF(\nu)$. This leads, via the formal calculation
\[
  \delta \VV^*(\nu) := \cK(\nu) \PPsi^*(\nu) = - \cK(\nu) D\cF(\nu) = \nu Q(\nu),
\]
to the formulation: For all $\varphi \in C_c^1([0,T]\times \cP(\cX))$ we have
\begin{equation}\label{e:Liouville:weak}
  \int_0^T \int_{\cP(\cX)} \bra{ \dot\varphi(t,\nu)-
  \skp{\nabla\varphi(t,\nu)}{\nu Q(\nu)}}
  \mathbb{{C}}(t,d\nu) \;dt = 0 \; .
\end{equation}
By the Lipschitz assumption on $Q$ the vector field $\nu Q(\nu)$ given by the components
$\bra{\nu Q(\nu)}_x = \sum_y \nu_y Q_{xy}(\nu)$ is also Lipschitz. Then standard theory implies that
equation~\eqref{e:Liouville:weak} has a unique solution (cf.~\cite[Chapter 8]{Ambrosio2008}).
\begin{definition}[Lifted action]
  Given $\MM\in \cP(\cP(\cX))$, $\VV:\cP(\cX)\to\R^{\cX\times\cX}$ and $\PPsi
  :\cP(\cX) \to \R^{\cX}$, we define the action of $(\MM,\VV)$ and
  $(\MM,\PPsi)$ by
  \begin{align*}
\vec\AA(\MM,\VV) &:= \int_{\cP(\cX)} \vec\cA(\nu,\VV(\nu)) \; \MM(d\nu)\;,\\
    \AA(\MM,\mathbb{\Psi}) &:= \int_{\cP(\cX)} \cA(\nu,\mathbb{\Psi}(\nu)) \; \MM(d\nu)\;.
 \end{align*}
\end{definition}
The next result tell us that is is sufficient to consider only
gradient vector fields. It is the analog of
Proposition~\ref{prop:Gradient fields}.
\begin{proposition}[Gradient fields for Liouville
  equation]\label{prop:Lio:GradientFields}
  If $(\CC,\VV)\in \vec\CCEE_T$ is a curve of finite action,
  then there exists $\PPsi:[0,T]\times \cP(\cX)\to \R^{\cX}$
  measurable such that $(\CC,\PPsi)\in \CCEE_T$ and
  \begin{equation}\label{e:Lio-action-vec}
    \int_0^T \AA(\CC(t),\PPsi(t)) dt  \leq     \int_0^T \vec\AA(\CC(t),\VV(t)) dt .
  \end{equation}
\end{proposition}
\begin{proof}
  Given a solution $(\CC,\VV)\in \vec\CCEE_T$, for each $t$ and
  $\nu\in\cP(\cX)$ we apply the contruction in the proof of
  Proposition~\ref{prop:Gradient fields} to $\VV(\nu)$ to
  obtain $\PPsi(t,\nu)$ with $\cA(\nu,\PPsi(t,\nu))\leq
  \vec\cA(\nu,\VV(t,\nu))$. It is readily checked that
  $(\CC,\PPsi)\in\CCEE_T$. Integration against $\CC$ and $dt$ yields
  \eqref{e:Lio-action-vec}.
\end{proof}
\begin{definition}[Lifted distance]
  Given $\MM,\NN \in \cP(\cP(\cX))$ we define
  \begin{equation}\label{e:Lio:W}
    \WW^2(\MM,\NN) := \inf\set{ \int_0^1 \AA(\mathbb{{C}}(t),\mathbb{\Psi}(t)) \; dt : (\mathbb{{C}},\mathbb{\Psi})\in \CCEE_1(\MM,\NN)}\;.
  \end{equation}
\end{definition}
Analogously to Remark \ref{rem:equiv-formulation} we obtain an equivalent formulation of $\WW$:
\begin{equation}\label{e:Lio:W-alt}
  \WW^2(\MM,\NN) = \inf\set{ \int_0^1 \vec\AA(\mathbb{{C}}(t),\VV(t)) \; dt : (\mathbb{{C}},\VV)\in \vec\CCEE_1(\MM,\NN)}\;.
\end{equation}
The following result is a probabilistic representation via
characteristics for the continuity equation. It is a variant of
\cite[Prop.~8.2.1]{Ambrosio2008} adapted to our setting.
\begin{proposition}\label{prop:Lio:representation}
For a given $\MM,\NN \in \cP(\cP(\cX))$ let
$(\mathbb{{C}},\mathbb{\Psi})\in \CCEE_T(\MM,\NN)$ be a solution of
the continuity equation with finite action.

Then there exists a probability measure~$\Theta$ on $\cP(\cX) \times
\AC([0,T],\cP(\cX))$ such that:
\begin{enumerate}
 \item Any $(\mu,{c}) \in \operatorname{supp} \Theta$ is a solution of the ODE
   \begin{align*}
     \dot{c}(t) &= \cK(c(t)) \mathbb{\Psi}(t,{c}(t))\quad \text{ for a.e. } t\in [0,T]\;,\\
     {c}(0) &= \mu\;.
   \end{align*}

 \item For any $\varphi\in C_b^0(\cP(\cX))$ and any $t\in [0,T]$ holds
\begin{equation}\label{e:Lio:disint}
 \int_{\cP(\cX)} \varphi(\nu) \; \mathbb{{C}}(t,d\nu) = \int_{\cP(\cX) \times \AC([0,T],\cP(\cX))} \varphi({c}(t)) \; \Theta(d\mu_0,d{c}) .
\end{equation}
\end{enumerate}
Conversely any $\Theta$ satisfying~\it{(1)} and
\begin{equation*}
 \int_{\cP(\cX) \times \AC([0,T],\cP(\cX))} \int_0^T  \cA\bra{c(t),\PPsi(t,c(t))}  dt \; \Theta(d\mu, d{c}) < \infty
\end{equation*}
induces a family of measures $\mathbb{{C}}(t)$
via~\eqref{e:Lio:disint} such that $(\mathbb{{C}},\mathbb{\Psi}) \in
\CCEE_T(\MM,\NN)$.
\end{proposition}
We will also use the measure $\bar\Theta$ on $\AC([0,T],\cP(\cX))$ given by
\begin{align}\label{e:Lio:disint:barTheta}
  \bar \Theta(d c)=\int_{\cP(\cX)}\Theta(d\mu,dc)\;.
\end{align}
Therewith, note that \eqref{e:Lio:disint} can be rewritten as the pushforward
$\CC(t)=(e_t)_\#\bar\Theta$ under the evaluation map $e_t:\AC([0,T],\cP(\cX))\ni c\mapsto
c(t)\in\cP(\cX)$ defined for any $\varphi\in C_b^0(\cP(\cX))$ by
\begin{equation*}
 \int_{\cP(\cX)} \varphi(\nu) \; \mathbb{{C}}(t,d\nu) = \int_{\AC([0,T],\cP(\cX))} \varphi({c}(t)) \; \bar\Theta(d{c}) .
\end{equation*}

\begin{proof}
  Let $(\CC,\PPsi)\in \CCEE_T(\MM,\NN)$ be a solution of the
  continuity equation with finite action. Define
  $\VV:[0,T]\times\cP(\cX)\to\R^{\cX\times\cX}$ via
  $\VV_{xy}(t,\nu)=w_{xy}(\nu)\nabla_{xy}\PPsi(t,\nu)$ and note that
  $\delta\VV(t,\nu)= \cK(\nu) \PPsi(t,\nu)$. We view $\cP(\cX)$ as a
  subset of $\R^{\cX}$ and $\delta\VV$ as a time-dependent vector
  field on $\R^\cX$ and note that $(\CC,\delta\VV)$ is a solution to
  the classical continuity equation in weak form
  \begin{align*}
      &\int_0^T \int_{\R^{\cX}} \bra{ \dot\varphi(t,\nu)- \nabla\varphi(t,\nu) \delta\VV(t,\nu) } \CC(t,d\nu) \;dt = 0
  \end{align*}
  for all $\varphi\in C^1_c\big((0,T)\times \R^{\cX}\big)$. Moreover,
  note that for any $\PPsi\in \R^\cX$ we have by Jensens inequality
  \begin{align*}
    |\cK(\nu)\PPsi|^2 &= \sum_{x\in\cX}\bigg|\sum_{y\in\cX}w_{xy}(\nu)(\PPsi_x-\PPsi_y)\bigg|^2\\
    &\leq \sum_{x,y\in\cX}C_w w_{xy}(\nu)\left|(\PPsi_x-\PPsi_y)\right|^2
                  = C_w \cA(\nu,\PPsi(\nu))\;,
  \end{align*}
  with
  \[
    C_w := \max_{x,y\in \cX} \sup_{\nu\in \cP(\cX)} w_{xy}(\nu) = \max_{x,y\in \cX} \sup_{\nu\in \cP(\cX)} \Lambda\!\bra{ \nu_x Q_{xy}(\nu) , \nu_y Q_{yx}(\nu) } .
  \]
  Since $Q:\cP(\cX)\to \R_+$ is continuous, $C_w$ is finite. This yields
  the integrability estimate
  \begin{align*}
    \int_0^T\int_{\R^d}|\delta\VV(t,\nu)|^2 \; d\CC(t,\nu)\;dt
    \leq
    C \int_0^T\AA(\CC(t),\PPsi(t))\; dt < \infty\;.
  \end{align*}
  Now, by the representation result
  \cite[Proposition~8.2.1]{Ambrosio2008} for the classical continuity
  equation there exists a probability measure~$\Theta$ on $\R^\cX \times
  \AC([0,T],\R^\cX)$ such that any $(\mu,{c}) \in \operatorname{supp}
  \Theta$ satisfies $c(0)=\mu$ and $\dot{c}(t) = \delta\VV(t,c(t))$ in
  the sense weak sense~\eqref{e:MF:continuity_equ:weak}
  and moreover, \eqref{e:Lio:disint} holds with
  $\cP(\cX)$ replaced by $\R^\cX$. Since, $\CC(t)$ is supported on
  $\cP(\cX)$ we find that $\Theta$ is actually a measure on $\cP(\cX)
  \times \AC([0,T],\cP(\cX))$, where absolute continuity is understood
  still w.r.t.~the Euclidean distance. To see that $\Theta$ is the
  desired measure it remains to check that for $\Theta$-a.e. $(\mu,c)$
  we have that $c$ is a curve of finite action. But this follows by
  observing that \eqref{e:Lio:disint} implies
  \begin{align*}
    \int\int_0^T\cA(c(t),\PPsi(t,c(t)))\; dt\; \Theta(d\mu,dc) =
    \int_0^T\AA(\CC(t),\PPsi(t))\; dt<\infty\;.
  \end{align*}
  This finishes the proof of the first statement.

  The converse, statement follows in the same way as in \cite[Proposition~8.2.1]{Ambrosio2008}.
\end{proof}
\begin{proposition}[Identification with Wasserstein distance]
  The distance $\WW$ defined in~\eqref{e:Lio:W} coincides with the
  $L^2$-Wasserstein distance on $\cP(\cP(\cX))$ w.r.t.~the distance
  $\cW$ on $\cP(\cX)$. More precisely, for $\MM,\NN\in \cP(\cP(\cX))$
  there holds
   \begin{equation*}
     \WW^2(\MM,\NN) = W_{\cW}^2(\MM,\NN) := \inf_{\GG\in \Pi(\MM,\NN)} \set{ \int_{\cP(\cX)\times\cP(\cX)} \cW^2(\mu,\nu) \ d\GG(\mu,\nu) } ,
   \end{equation*}
   where $\Pi(\MM,\NN)$ is the set of all probability measures on
   $\cP(\cX) \times \cP(\cX)$ with marginals $\MM$ and $\NN$.
\end{proposition}
\begin{proof}
  We first show the inequality ``$\geq$''. For $\eps>0$ let
  $(\CC,\PPsi)$ be a solution to the continuity equation such that
  $\int_0^1\AA(\CC(t),\PPsi(t))dt\leq \WW^2(\MM,\NN)+\eps$ and let
  $\bar\Theta$ be the measure on $\AC([0,T],\cP(\cX))$ given by the
  previous Proposition. Then we obtain a coupling $\GG\in\Pi(\MM,\NN)$
  by setting $\GG=(e_0,e_1)_\#\bar\Theta$. This yields
  \begin{align*}
    W_{\cW}^2(\MM,\NN) &\leq \int \cW^2(\mu,\nu) \ d\GG(\mu,\nu) = \int
    \cW^2(c(0),c(1))d\bar\Theta\\
    & \leq
    \int\int_0^1\cA(c(t),\PPsi(t,c(t)))dt d\bar\Theta(c)=\int_0^1\AA(\CC(t),\PPsi(t))dt \leq \WW^2(\MM,\NN)+\eps\;.
  \end{align*}
  Since $\eps$ was arbitrary this yields the inequality ``$\geq$''.

  To prove the converse inequality ``$\leq$'', fix an optimal coupling
  $\GG$, fix $\eps>0$ and choose for $\GG$-a.e.~$(\mu,\nu)$ a couple
  $(c^{\mu,\nu},v^{\mu,\nu})\in\vec\CE_1(\mu,\nu)$ such that
  \begin{equation*}
    \int_0^1\vec\cA(c^{\mu,\nu}(t),v^{\mu,\nu}(t))dt\leq \cW(\mu,\nu) +\eps.
  \end{equation*}
  Now, define a family of measures $\CC:[0,1]\to\cP(\cP(\cX))$
  and a family of vector-valued measures
  $V:[0,1]\to\cP(\cP(\cX);\R^{\cX\times\cX})$ via
  \begin{align*}
    &d\CC(t,\tilde{\nu}) = \int_{\cP(\cX)\times\cP(\cX)} d\delta_{c^{\mu,\nu}(t)}(\tilde{\nu})d\GG(\mu,\nu)\;,\\
   &V(t,\tilde{\nu})=\int_{\cP(\cX)\times\cP(\cX)}v^{\mu,\nu}(t)d\delta_{c^{\mu,\nu}(t)}(\tilde{\nu}) d\GG(\mu,\nu)\;.
  \end{align*}
  Note that $V(t)\ll\CC(t)$ and define
  $\VV:[0,1]\times\cP(\cX)\to\R^{\cX\times\cX}$ as the density of $V$
  w.r.t.~$\CC$. By linearity of the continuity equations have that
  $(\CC,\VV)\in\vec\CCEE_1(\MM,\NN)$. Moreover, we find
 \begin{align*}
    \int_0^1\vec\AA(\CC(t),\VV(t))dt &= \int_0^1\int \vec\cA(c^{\mu,\nu}(t),v^{\mu,\nu}(t))d\GG(\mu,\nu)dt\\
    &\leq \int \cW^2(\mu,\nu) d\GG(\mu,\nu) +\eps = W^2_\cW(\MM,\NN) +\eps\;.
  \end{align*}
  Since $\eps$ was arbitrary, in view of \eqref{e:Lio:W-alt} this
  finishes the proof.
\end{proof}
Finally, we can obtain a gradient flow structure for the Liouville
equation \eqref{e:Liouville} in a straightforward manner by averaging
the gradient flow structure of the underlying dynamical system.

To this end, given $\MM\in \cP(\cP(\cX))$ define the free energy by
 \begin{equation*}
   \FF(\MM) := \int_{\cP(\cX)} \cF(\nu) \; \MM(d\nu)\;,
 \end{equation*}
 and define the Fisher information by
 \begin{equation*}
   \II(\MM) := \AA(\MM,-D\FF) = \int_{\cP(\cX)} \cI(\nu) \; \MM(d\nu) .
 \end{equation*}
\begin{proposition}[Gradient flow structure for Liouville equation]\label{prop:Lio:GF}
  The Liouville equation \eqref{e:Liouville} is the gradient flow of
  $\FF$ w.r.t.~$\WW$. Moreover precisely, $\sqrt{\II}$ is a strong
  upper gradient for $\FF$ on the metric space $(\cP(\cP(\cX)),\WW)$
  and the curves of maximal slope are precisely the solutions to
  \eqref{e:Liouville}. In other words, for any absolutely continuous
  curve $\mathbb{{C}}$ in $\cP(\cP(\cX))$ holds
 \begin{equation}\label{e:Lio:LDfunctional}
   \JJ(\CC) := \FF(\mathbb{{C}}(T)) - \FF(\mathbb{{C}}(0)) + \frac{1}{2} \int_0^T \II(\mathbb{{C}}(t)) \; dt + \frac{1}{2} \int_0^T \AA(\mathbb{{C}}(t),\mathbb{\Psi}(t)) \; dt \geq 0
 \end{equation}
 with $(\CC(t),\mathbb{\Psi}(t))\in \CCEE_T$. Moreover, $\JJ(\CC)=0$
 if and only if $\mathbb{{C}}$ solves~\eqref{e:Liouville:weak}.
\end{proposition}
\begin{proof}
  Let $\bar\Theta$ be the disintegration of $\CC$
  from Proposition~\ref{prop:Lio:representation}
  defined in~\eqref{e:Lio:disint:barTheta}. The fact that $\sqrt{\II}$ is a strong
  upper gradient of $\FF$ can be seen by integrating its defining inequality on the underlying level~\eqref{e:strongupper} w.r.t.~$\bar\Theta$
  \begin{align*}
    \abs{\FF(\CC(t_2)) - \FF(\CC(t_1))} &\leq \int_{\AC([0,T];\cP(\cX))}
    \abs{\cF(c(t_2))-\cF(c(t_1))} \; \bar\Theta(dc) \\
    &\leq \int_{\AC([0,T];\cP(\cX))} \int_{t_1}^{t_2}\sqrt{\cI(c(t))}\sqrt{\cA(c(t),\psi(t))} \; dt \; \bar\Theta(dc).
  \end{align*}
  Then, using Jensen's inequality on the concave function $(a,b)\mapsto
  \sqrt{ab}$ , we get the strong upper gradient property for $\sqrt{\II}$.

  The ``if'' part of the last claim is easily verified from the definition.
  Now, assume that $\JJ(\CC) = 0$. Since $\CC$ is absolutely continuous, we can apply Proposition~\ref{prop:Lio:representation} and obtain the probabilistic representation $\bar\Theta\in \cP\bra{\AC([0,T]\times \cP(\cX))}$~\eqref{e:Lio:disint:barTheta} such that $\CC(t) = (e_t)_\# \bar\Theta$. Then, \eqref{e:Lio:LDfunctional} can be obtained by just
  integrating $\cJ$ from \eqref{e:LDfunctional} along $\bar\Theta$ and it holds $\cJ(c) = 0$ for $\bar\Theta$-a.e.~$c\in \AC([0,T],\cP(\cX))$. These $c$ are by Proposition
  \ref{prop_max_slope_nonlin} solutions to~\eqref{e:GF:KDF} and satisfy
  $c(t)\in \cP^*(\cX)$ for all $t>0$. Then, we can conclude by the converse statement of Proposition~\ref{prop:Lio:representation} that $\cK(c(t)) \PPsi(t,c(t)) = - \cK(c(t)) D\cF(c(t))$, which implies since $c(t)\in\cP^*(\cX)$ for $t>0$ up to a constant that $\PPsi(t,c(t)) = - D\cF(c(t))$ and hence $\CC$ solves~\eqref{e:Liouville:weak}.
\end{proof}
\section{From weakly interacting particle systems to mean field systems}\label{S:Limit}
In this section, we will show how the gradient flow structure we
described in the previous sections arises as the limit of gradient
flow structures for $N$-particle systems with mean field interactions,
in the limit $N \to \infty$. Moreover, we show that the empirical
distribution of the $N$-particle dynamics converges to a solution of
the non-linear equation \eqref{eq_evol_nonlin}.
\begin{notation}
  For $N$ an integer bold face letters are elements connected to the space $\cX^N$ and hence implicitly depending on $N$. Examples are vectors $\bx,\by\in \cX^N$, matrices $\bm{Q}\in \R^{\cX^N\times \cX^N}$ or measures $\bm\mu\in \cP(\cX^N)$. For $i \in \set{1,\dots, N}$ let $\be^i$ be the placeholder for $i$-th particle, such that $\bx \cdot \be^i = x_i\in \cX$ is the position of the $i$-th particle. For $\bx  \in \cX^N$ and $y\in \cX$ we denote by $\bx^{i;y}$ the particle system obtained from $\bx$ where the $i$-th particle jumped to site $y$
  \[
    \bx^{i;y} := \bx - (x_i -y) \be^i = (x_{1} ,\dots,x_{i-1},y,x_{i+1},\dots,x_{N}) .
  \]
  $L^N(\bx)$ will denote the empirical distribution for $\bx\in \cX^N$, defined by
  \begin{equation}\label{e:def:EmipiricalDistribution}
    L^N(\bx) := \frac{1}{N} \sum_{i=1}^N \delta_{x_i}
  \end{equation}
  We introduce the discretized simplex $\cP_N(\cX)\subset \cP(\cX)$, given by
  \begin{equation*}
    \cP_N(\cX) := \set{ L^N(\bx) : \bx \in \cX^N} .
  \end{equation*}
\end{notation}
Let us introduce a natural class of mean-field dynamics for the $N$-particle system.
We follow the standard procedure outlined in Remark~\ref{rem:MarkovKernel}.

In analog to Definition~\ref{def:GibbsPotential}, we fix $K:\mathcal{P}(\mathcal{X})\times\mathcal{X}\rightarrow\R$ such that for each $x\in \mathcal{X}, K_{x}$ is a twice continuously
differentiable function on $\mathcal{P}(\mathcal{X})$ and set $U(\mu):= \sum_{x\in \cX} \mu_x K_x(\mu)$.
For every natural number $N$ define the probability measure $\bm\pi^N$ for $\bm x\in \cX^N$ by
\begin{equation}\label{e:def:piN}
  \bm{\pi}^{N}_{\bm{x}} := \frac{1}{\bm Z^{N}}\exp\left(-N U\bra{L^N \bm x}\right),
\end{equation}
and $\bm Z^{N} := \sum_{\bm{x}\in\mathcal{X}^{N}}\exp\bra{-N U\bra{L^N \bm x}}$ is the partition sum. This shall be the invariant measure of the particle system and is already of mean-field form.

To introduce the dynamics, we use a family $\set{A^N(\mu)\in \R^{\cX\times \cX}}_{\mu \in \cP_N(\cX)}$ of irreducible symmetric matrices and define the rate matrices $\set{Q^{N}(\mu)\in \R^{\cX \times \cX}}_{\mu\in \cP_N(\cX)}$ for any $\bm x\in \cX^N, y\in \cX, i\in\set{1,\dots N}$ by
\begin{align}\label{e:def:QN}
  Q^N_{x_i,y}(L^N \bm x) &:= \sqrt{\frac{\bm \pi^N_{\bm x^{i;y}}}{\bm \pi^N_{\bm x}}}\  A_{x,y}^N(L^N\bm x)\\
  &= \exp\bra{-\tfrac{N}{2}\bra{U(L^N \bm x^{i;y})-U(L^N \bm x)}} A_{x,y}^N(L^N\bm x) .\notag
\end{align}
Finally, the actual rates of the $N$-particle system are given in terms of the rate
  matrix
\begin{equation}\label{e:def:bQ}
  \bm Q^{N}\in \R^{\cX^N \times \cX^N}: \qquad \bm{Q}^{N}_{\bx,\bx^{i;y}}:= Q_{x_i ,y}^N(L^N(\bx)) .
\end{equation}
By construction $\bm Q^N$ is irreducible and reversible w.r.t.~the unique invariant measure~$\bm \pi^N$.
\begin{remark}
  The irreducible family of matrices $\set{A^N(\mu)}_{\mu\in\cP(\cX)}$ encodes the underlying
  graph structure of admissible jumps and also the rates of the jumps. For instance,
  $A^N_{x,y}(\nu) = \alpha_{x,y}$ for any symmetric adjacency matrix $\alpha \in \set{0,1}^{\cX\times\cX}$ corresponds to Glauber dynamics on the corresponding graph. Another choice is $A^N_{x,y}(L^N \bm x) := \exp\bra{-\tfrac{N}{2}\abs{U(L^N \bm x)- U(L^N \bm x^{i;y})}}$, which corresponds to Metropolis dynamics on the complete graph. In particular, all of these examples satisfy Assumption~\ref{assume:A}.
\end{remark}
\begin{assume}[Lipschitz assumptions on rates]\label{assume:A}
  There exists a family of irreducible symmetric matrices $\set{A(\mu)}_{\mu\in \cP(\cX)}$ such that  $\mu \mapsto A(\mu)$ is Lipschitz continuous on $\cP(\cX)$ and the family $\set{A^N(\mu)}_{\mu \in \cP(\cX), N\in\N}$ of irreducible symmetric matrices satisfies
  \begin{equation*}
    \forall x,y\in \cX: \qquad A^N_{x,y} \to A_{x,y} \qquad\text{on } \cP(\cX) \text{ as } N\to \infty .
  \end{equation*}
\end{assume}
\begin{lemma}\label{lem:UniConvGammaN}
  Assume $\set{Q^{N}(\mu)\in \R^{\cX \times \cX}}_{\mu\in \cP_N(\cX)}$
  is given by~\eqref{e:def:QN} with $A^N$ satisfying Assumption~\ref{assume:A}, then for all $x,y\in \cX$
\begin{equation}\label{ass:UniConvGammaN}
   Q^N_{x,y} \to Q_{x,y} \qquad \text{on $\cP(\cX)$}
\end{equation}
with $Q_{x,y}(\mu) = \sqrt{\frac{\pi_y(\mu)}{\pi_x(\mu)}} A_{x,y}(\mu)$ with $\pi$ given in~\eqref{e:def:piH}.
In particular, $\mu\mapsto Q_{x,y}(\mu)$ is Lipschitz continuous on $\cP(\cX)$ for all $x,y\in \cX$.
\end{lemma}
\begin{proof}
By~\cite[Lemma 4.1]{Budhiraja2014a} holds for $\bm x\in \cX^N$ with $\mu=L^N \bm x$, $y\in \cX$ and $i\in \set{1,\dots, N}$
\begin{align*}
 \frac{\bm \pi^N_{\bm x^{i;y}}}{\bm \pi^N_{\bm x}} &=  N\bra{U(L^N\bm x) - U(L^N\bm x^{i;y})} = \partial_{\mu_x}U(\mu) - \partial_{\mu_y}U(\mu) + O(N^{-1}) \\
 &= \frac{\pi_y(\mu)}{\pi_x(\mu)} + O(N^{-1}),
\end{align*}
which shows by Assumption~\ref{assume:A} the convergence statement. The Lipschitz continuity follows, since $A$ is assumed Lipschitz and
the function $\mu \mapsto \partial_{\mu_x} U(\mu) = \mu_x + \sum_{y} \mu_y \partial_{\mu_x} K_y(\mu)$ is continuously differentiable, since $K$ is assumed twice continuously differentiable.
\end{proof}
\begin{remark}
  The mean-field behavior is manifested in the convergence
  statement \eqref{ass:UniConvGammaN}.  The typical example we have in
  mind, as presented in Section 4 of \cite{Budhiraja2014a}, is as
  follows: the mean-field model is described by
  $$K_{x}(\mu) := V(x) + \underset{y}{\sum} \hspace{1mm} W(x,y)\mu_y$$
  where $V$ is a potential energy and $W$ an interaction energy
  between particles on sites $x$ and $y$. For the $N$ particle system,
  we can use a Metropolis dynamics, where possible jumps are those
  between configurations that differ by the position of a single
  particle, and reversible with respect to the measure
  $$\mathbf{\pi}^N_{\bx} = A^{-1}\exp(-U^N(\bx)); \hspace{3mm} U^N(\bx) := \underset{i=1}{\stackrel{N}{\sum}} \hspace{1mm} V(x_i) + \frac{1}{N}\underset{i=1}{\stackrel{N}{\sum}}\underset{j=1}{\stackrel{N}{\sum}} \hspace{1mm} W(x_i,x_j).$$
  Note, the by the definition of the $L^N$, we have the identity
  \[
    U^N(\bm x) = N U(L^N\bm x) \quad\text{with}\quad U(\mu) = \sum_{x} \mu_x K_x(\mu),
  \]
  which makes it consistent with Definition~\ref{e:def:piN}.
  This is a typical class of mean-field spin systems from statistical mechanics.

  In particular, the Curie-Weiss mean-field spin model for
  ferromagnetism is obtained by choosing $\cX=\set{-,+}$,
  $V(-)=V(+)=W(-,-)=W(+,+)= 0$ and $W(-,+)=W(+,-)=\beta>0$. This is
  among the simplest models of statistical mechanics showing a phase
  transition in the free energy
  \[
    \cF_\beta(\mu) := \sum_{\sigma\in \set{-,+}}\bra{\log\mu_\sigma + K_\sigma(\mu)}\mu_\sigma = \mu_{-} \log \mu_{-} + \mu_+ \log \mu_+ + 2\beta \mu_{-} \mu_+
  \]
  at $\beta=1$. For $\beta \leq 1$ the free energy is convex whereas
  for $\beta> 1$ it is non-convex on~$\cP(\cX)$. We will investigate
  this phase transition on the level of curvature for the mean-field
  system as well as for the finite particle system in future work.
\end{remark}
\subsection{Gradient flow structure of interacting \texorpdfstring{$N$}{N}-particle systems}
The $N$-particle dynamics on $\cX^N$ is now defined by the rate matrix $\bm{Q}^{N}$ given as in~\eqref{e:def:bQ} with the generator
\begin{equation}\label{e:def:PartGenerator}
 \mathcal{L}^{N}f := \sum_{i=1}^N \sum_{y\in \cX} (f({\bm{x}}^{i;y}) - f({\bm{x}}))\bm{Q}^N_{{\bm{x}}, {\bm{x}}^{i;y}}  .
\end{equation}
Likewise the evolution of an initial density $\bm\mu_0 \in \cP(\cX^N)$ satisfies
\begin{equation}\label{e:def:PartEvol}
 \dot{\bm{c}}(t) = \bm{c}(t) \bm{Q}^{N}   .
\end{equation}
Since by construction the rate matrix $\bm Q^N$ defined in~\eqref{e:def:bQ}
satisfy the detailed balance condition w.r.t.~$\bm\pi^N$~\eqref{e:def:piN},
this is the generator of a reversible Markov process w.r.t.~$\bm\pi^N$ on the finite space
$\cX^N$. Hence, we can use the framework developed in \cite{Maas2011} and
\cite{Mielke2013} to view this dynamics as a gradient flow of the
relative entropy with respect to its invariant measure. Let us
introduce the relevant quantities.

We define the relative entropy $\mathcal{H}(\bm{\mu}\mid\bm{\pi}^N)$
for $\bm{\mu},\bm{\pi}^N\in \mathcal{P}(\mathcal{X}^{N})$ by setting
\begin{equation*}
 \bm\cF^{N}(\bm\mu) := \cH(\bm\mu \mid \bm\pi^N) = \sum_{\bx\in \cX^N} \bm\mu_{\bx} \log \frac{ \bm\mu_{\bx} }{\bm\pi^N_{\bx}}\;.
\end{equation*}
Furthermore we define the action of a pair $\bm{\mu}\in\cP(\cX^N)$,
$\bm{\psi}\in\R^{\cX^N}$ by
\begin{equation*}
\bm{\cA}^{N}(\bm{\mu},\bm{\psi})=\frac{1}{2}\sum_{\bm{x},\bm{y}}(\bm\psi_{\bm{y}}-\bm\psi_{\bm{x}})^{2}\bm{w}^{N}_{\bm{x},\bm{y}}(\bm{\mu}),
\end{equation*}
where the weights $\bm{w}^N_{\bx,\by}(\bm\mu)$ are defined like in~\eqref{e:MF:weights} as follows
\begin{equation}\label{e:PS:weights}
  \bm{w}^N_{\bx,\by}(\bm\mu) := \Lambda\!\bra{\bm\mu_{\bx} \bm Q^N(\bx,\by) , \bm\mu_{\by} \bm Q^N(\by,\bx)} .
\end{equation}
Then, a distance $\bm{\cW}^N$ on $\cP(\cX^N)$ is given by
\begin{equation}\label{e:particle:def:W}
  \bm\cW^{N}(\bm\mu, \bm\nu)^2 := \underset{(\bm{c}(t), \bm\psi(t))}{\inf} \hspace{1mm} \int_0^1{\bm{\cA}^{N}(\bm{c}(t),  \bm\psi(t))dt}
\end{equation}
where the infimum runs over all pairs such that $\bm{c}$ is a path
from $\bm\mu$ to $\bm\nu$ in $\mathcal{P}(\cX^N)$, and such that the
continuity equation
\begin{equation}\label{e:particles:CE}
  \dot{\bm{c}}_{\bx}(t) + \sum_{\by} (\bm\psi_{\by}(t) - \bm\psi_{\bx}(t))\bm{w}^{N}_{\bx,\by}(\bm{{c}}(t))= 0
\end{equation}
holds. For details of the construction and the proof that this defines
indeed a distance we refer to \cite{Maas2011}. In particular, we note
that for any absolutely continuous curve
$\bm{c}:[0,T]\to(\cP(\cX^N),\bm\cW^N)$ there exist a function
$\bm\psi:[0,T]\to\R^{\cX^N}$ such that the continuity equation
\eqref{e:particles:CE} holds.

Finally, we define the $N$-particle Fisher information by
\begin{align*}
  \bm\cI^{N}(\bm\mu) :=\begin{cases} \frac{1}{2} \sum\limits_{(\bm{x},\bm{y})\in E_{\bm{\mu}}} \bm{w}^{N}_{\bm{xy}}(\bm{\mu})(\log(\bm{\mu}_{\bm{x}}\bm{Q}^{N}_{\bm{xy}}(\bm{\mu}))-\log(\bm{\mu}_{\bm{y}}\bm{Q}^{N}_{\bm{yx}}(\bm{\mu})))^{2} & \bm{\mu}\in \cP^{*}(\cX^{N})\\
    \infty. & \text{otherwise}\end{cases}
\end{align*}
We formulate the statement that \eqref{e:def:PartEvol} is the gradient
flow of $\bm\cF^N$ w.r.t.~$\bm\cW^N$ again in terms of curves of maximal
slope.
\begin{proposition}
  For any absolutely continuous curve $\bm{{c}} : [0,T] \rightarrow
  (\mathcal{P}(\cX^N),\bm\cW^N)$ the function $\bm{\cJ}^{N}$ given by
\begin{align}\label{e:particle:LDfunctional}
  \bm{\cJ}^{N}(\bm{{c}}) :=\bm{\cF}^{N}(\bm{{c}}(T)) -
  \bm{\cF}^{N}(\bm{{c}}(0)) +
  \frac{1}{2}\int_0^T \bm{\cI}^{N}(\bm{{c}} (t))+\bm{\cA}^{N}(\bm{{c}}(t),\bm{\psi}(t)) \; dt
\end{align}
is non-negative, where $\bm{\psi}_{t}$ is such that the continuity
equation~\eqref{e:particles:CE} holds. Moreover, a curve $\bm{{c}}$
is a solution to $\dot{\bm{{c}}}(t) =\bm{{c}}(t) \bm{Q}^{N}$ if
and only if $\bm\cJ^N(\bm c)=0$.
\end{proposition}
\begin{proof}
  The proof is exactly the same as for Proposition
  \ref{prop_max_slope_nonlin}, so we omit it.
\end{proof}
\subsection{Convergence of gradient flows}
In this section we prove convergence of the empirical distribution of
the $N$-particle system \eqref{e:def:PartEvol} to a solution of the
non-linear equation \eqref{eq_evol_nonlin}. This will be done by using
the gradient flow structure exibited in the previous sections together
with the techniques developed in \cite{Serfaty2011} on convergence of
gradient flows.

Heuristiclly, consider a sequence of gradient flows associated to a
senquence of metric spaces and engergy functionals. Then to prove
convergence of the flows it is sufficient to establish convergence of
the metrics and the energy functionals in the sense that functionals
of the type \eqref{e:particle:LDfunctional} satisfy a suitable notion
of $\Gamma-\liminf$ estimate.

In the following Theorem \ref{thm_serfaty} we adapt the result in
\cite{Serfaty2011} to our setting.

We consider the sequence of metric spaces $\bS^{N} :=
(\cP(\cX^N),\bm\cW^{N})$ with $\bm\cW^{N}$ defined in~\eqref{e:particle:def:W}
and the limiting metric space $\SS:=\bra{\cP(\cP(\cX)), \WW}$ with
$\WW$ defined in~\eqref{e:Lio:W}.
The following notion of convergence will provide the correct topology
in our setting.
\begin{definition}[Convergence of random measures]\label{def:tauConv}
  A sequence $\bm{\mu}^{N} \in \mathcal{P}(\mathcal{X}^{N})$ converges
  in $\tau$ topology to a point $\MM
  \in\mathcal{P}(\mathcal{P}(\mathcal{X}))$ if and only if
  $L^N_\#(\bm{\mu}^{N})\in \mathcal{P}(\mathcal{P}(\mathcal{X}))$
  converges in distribution to $\MM$, where $L^N : \cX^N \to \cP(\cX)$
  is defined in~\eqref{e:def:EmipiricalDistribution}.
  Likewise, for $\bra{\bm{c}^N(t)}_{t\in [0,T]}$ with $\bm{c}_t^N \in \cP(\cX^N)$: $\bm{c}^N \stackrel{\tau}{\to} \CC$
  if for all $t\in [0,T]$, $L^N_\# \bm{c}^N(t) \rightharpoonup \CC(t)$.
\end{definition}
\begin{theorem}[{Convergence of gradient flows \`a la \cite{Serfaty2011}}]\label{thm_serfaty}
  Assume there exists a topology~$\tau$ such that whenever a sequence
  $\bm{c}^N \in \AC([0,T],\bS^{N})$ converges
  pointwise w.r.t.~$\tau$ to a
  limit $\mathbb{{C}} \in \AC([0,T],\SS)$, then this convergence is
  compatible with the energy functionals, that is
\begin{equation}\label{e:Serfaty:energy}
   \bm{c}^N \stackrel{\tau}{\to} \mathbb{{C}} \qquad\Rightarrow\qquad \liminf_{N\to\infty} \frac{1}{N} \bm\cF^{N}(\bm{c}^N(T)) \geq \FF(\mathbb{{C}}(T)) - \cF_0 ,
\end{equation}
  for some finite constant $\cF_0\in \R$. In addition, assume the following holds
\begin{enumerate}
\item $\liminf$-estimate of metric derivatives:
\begin{equation}\label{e:Serfaty:derivs}
 \liminf_{N\to\infty} \frac{1}{N}\int_0^T  \bm\cA^{N}(\bm{c}^N(t), \bm\psi^N(t)) \; dt  \geq  \int_0^T \AA(\mathbb{{C}}(t),\mathbb{\Psi}(t)) \; dt ,
\end{equation}
where $(\bm{c}^N,\bm\psi^N)$ and $(\mathbb{{C}}(t),\mathbb{\Psi}(t))$ are related via the respective continuity equations in $\bS^{N}$ and $\SS$.
\item $\liminf$-estimate of the slopes pointwise in $t\in [0,T]$:
\begin{equation}\label{e:Serfaty:slopes}
\liminf_{N\to \infty} \frac{1}{N} \bm\cI^{N}(\bm{c}^N(t)) \geq \II(\mathbb{{C}}(t)) .
\end{equation}
\end{enumerate}
Let $\bm{c}^N$ be a curve of maximal slope on $(0,T)$ for $\bm\cJ^{N}$~\eqref{e:particle:LDfunctional} such that $\bm{c}^N(0) \stackrel{\tau}{\to} \mathbb{C}(0)$ which is well-prepared in the sense that  $\lim_{N\to \infty}\bm\cF^{N}(\bm{c}^N(0))=\FF(\mathbb{C}(0))$.
Then $\mathbb{{C}}$ is a curve of maximal slope for $\JJ$~\eqref{e:Lio:LDfunctional} and
\begin{align*}
  \forall t\in[0,T), \quad \lim_{N\to \infty} \frac{1}{N} \bm\cF^{N}(\bm{c}^N(t)) &= \FF(\mathbb{{C}}(t)) \\
  \frac{1}{N} \bm\cA^{N}(\bm{c}^N, \bm\psi^N) &\to   \AA(\mathbb{{C}},\mathbb{\Psi})\quad\text{in}\quad L^{2}[0,T] \\
  \frac{1}{N} \bm\cI^{N}(\bm{c}^N) &\to \II(\mathbb{C}) \quad \text{in} \quad L^{2}[0,T]
\end{align*}
\end{theorem}
\begin{proof}
 Let us sketch the proof. The assumptions~\eqref{e:Serfaty:energy},~\eqref{e:Serfaty:derivs},~\eqref{e:Serfaty:slopes} and the well-preparedness of the initial data allow to pass in the limit in the individual terms of $\frac{1}{N}\bm\cJ^{N}$~\eqref{e:particle:LDfunctional} to obtain
 \begin{equation*}
   \liminf_{N\to \infty} \frac{1}{N} \bm\cJ^{N}(\bm{c}^N) \geq \JJ(\mathbb{{C}}) .
 \end{equation*}
 Hence, if each $\bm{c}^N$ is a curve of maximal slope w.r.t.~$\frac{1}{N} \bm\cJ^{N}(\bm{c}^N)$ then so is $\mathbb{{C}}$ w.r.t.~$\JJ$.
The other statements also can be directly adapted from \cite{Serfaty2011}.
\end{proof}
\subsection{Application}
To apply Theorem~\ref{thm_serfaty}, we first have to show the convergence of the energy~\eqref{e:Serfaty:energy}.
\begin{proposition}[$\liminf$ inequality for the relative entropy]\label{prop:liminf:Ent}
Let a sequence $\bm{\mu}^N \in \cP(\bm{\cX}^N)$ be given such that $\bm{\mu}^{N}\stackrel{\tau}{\to} \MM$ as $N\to \infty$, then
\begin{equation*}
 \liminf_{N\rightarrow \infty}\frac{1}{N}\mathcal{H}(\bm{\mu}^{N} \mid \bm{\pi}^{N})\geq
\int_{\mathcal{P}(\mathcal{X})}(\mathcal{F}(\nu) - \cF_0) \; \MM(d\nu) = \FF(\MM) - \cF_0,
\end{equation*}
where
\begin{equation*}
 \cF_0 := \inf_{\mu\in \cP(\cX)} \cF(\mu) .
\end{equation*}
\end{proposition}
\begin{proof}
  First we note that the relative entropy can be decomposed as
  follows:
  \begin{align*}
    \cH(\bm\mu^N\mid\bm\pi^N) = &\int \cH\big(\bm\mu^N(\,\cdot\mid L^N=\nu)\;\big|\; \bm\pi^N(\,\cdot\mid L^N=\nu)\big) \, d L^N_\# \bm\mu^{N}(\nu)\\
                               &+ \cH(L^N_\# \bm\mu^{N}\mid L^N_\# \bm\pi^{N})
  \end{align*}
  Let us denote by $\bm M^N$ the uniform probability measure on
  $\cP_N(\cX)$. Using the fact that relative entropy w.r.t.~a
  probability measure is non-negative we arrive a the estimate
 \begin{align}\label{eq:ent-est1}\nonumber
    \cH(\bm\mu^N\mid\bm\pi^N) &\geq \cH(L^N_\# \bm\mu^{N}\mid L^N_\# \bm\pi^{N})\\\nonumber
                              &= \cH(L^N_\# \bm\mu^{N}\mid \bm M^N) + \Expect_{L^N_\# \bm\mu^{N}}\!\left[\log\frac{d\bm M^N}{dL^N_\# \bm\pi^{N}}\right]\\
&\geq \Expect_{L^N_\# \bm\mu^{N}}\!\left[\log\frac{d\bm M^N}{dL^N_\# \bm\pi^{N}}\right]\;.
\end{align}
For $\nu \in \cP_N(\cX)$ we set $\cT_N(\nu) = \set{\bx \in \cX^N :
  L^N(\bx) = \nu }$. Then by the definition
of~$\bm\pi^N$~\eqref{e:def:piN} and
$U(\nu)$~\eqref{e:def:MF:FreeEnergy} we have
\begin{align*}
  L^N_\# \bm\pi^{N}(\nu) = \frac{|\cT_N(\nu)|}{\bm Z^N}\exp\big(-NU(\nu)\big)\;.
\end{align*}
From \eqref{eq:ent-est1} we thus conclude
\begin{align}\nonumber
  \frac1N \cH(\bm\mu^N\mid\bm\pi^N) \geq &-\frac1N \log |\cP_N(\cX)| +\frac1N \log \bm Z^N\\\label{eq:ent-est2}
  &-\frac1N \Expect_{L^N_\# \bm\mu^{N}}\!\big[\log|\cT_N|\big] + \Expect_{L^N_\# \bm\mu^{N}}[U]\;.
\end{align}
The cardinality of $\cP_N(\cX)$ is given by $\binom{N + d - 1}{N} \leq N^{d-1}/d!$ and hence
\begin{align}\label{eq:ent-est3}
  \log | \cP_N(\cX)| \leq (d-1)\log N\;.
\end{align}
Moreover, by Stirling's formula (cf.~Lemma~\ref{lem:stirling}), it follows that for any $\nu\in \cP_N(\cX)$
\begin{align}\nonumber
  -\frac{1}{N} \log \abs{\cT_N(\nu)} &= - \frac{1}{N} \log
  \frac{N!}{\prod_{x\in \cX} (N\nu(x))!}\\\label{eq:ent-est4} &= \sum_{x\in \cX} \nu(x)
  \log \nu(x) + O\left(\frac{\log N}{N}\right)\;.
\end{align}
Furthermore, we have
\begin{equation*}
  \bm Z^N = \sum_{\nu\in \cP_N(\cX)} e^{-N U(\nu)} |\cT_N(\nu)| .
\end{equation*}
Hence, by using Sanov's and Varadhan's theorem \cite{Dupuis} on the
asymptotic evaluation of exponential integrals, it easily follows that
\begin{align}\label{eq:ent-est5}
  \lim_{N\rightarrow \infty} \frac{1}{N}\log \bm Z^N=-\inf_{\nu\in
    \cP(\cX)} \cF(\nu) =: - \cF_0 .
\end{align}
Combing now \eqref{eq:ent-est2} with
\eqref{eq:ent-est3}-\eqref{eq:ent-est5} finishes the proof.
\end{proof}
The other ingredient of the proof of Theorem~\ref{thm_serfaty} consists in proving the convergence of the metric derivatives~\eqref{e:Serfaty:derivs} and slopes~\eqref{e:Serfaty:slopes}.
\begin{proposition}[Convergence of metric derivative and slopes]\label{prop:liminf:MDslopes}
 Let $\bm{{c}}^N$ be an element of $ \AC([0,T],\cP(\cX^N)),$ and choose $\bm{\psi}^{N}:[0,T]\rightarrow \cP(\cX^N)$ such that $(\bm{{c}}^{N},\bm{\psi}^N)$ solves the continuity equation. Furthermore, assume that
\begin{equation*}
  \bm{c}^N \stackrel{\tau}{\to} \mathbb{{C}},
\end{equation*}
with some measurable $\mathbb{{C}}: [0,T]\rightarrow \cP(\cP(\cX)),$
and that \[\liminf_{N\to \infty} \int_0^T \frac{1}{N} \bm{\cA}^{N}(\bm{{c}}^{N}(t),\bm{\psi}^N(t)) dt <\infty.\]
 Then  $\mathbb{C} \in \AC\bra{[0,T],\cP(\cP(\cX))},$ and there exists $\mathbb{\Psi}:[0,T]\rightarrow \cP(\cP(\cX)),$ for which $(\mathbb{C},\mathbb{\Psi})$ satisfy the continuity equation and for which we have
\begin{equation*}
 \liminf_{N\to \infty} \int_0^T \frac{1}{N} \bm{\cA}^{N}(\bm{{c}}^{N}(t),\bm{\psi}^N(t)) dt \geq \int_0^T \AA(\mathbb{{C}}(t),\PPsi(t)) \; dt
\end{equation*}
 and
\begin{equation*}
 \liminf_{N\to \infty} \int_0^T \frac{1}{N} \bm{\cI}^{N}\left(\bm{{c}}^{N}(t)\right) dt \geq \int_0^T  \II\left(\mathbb{{C}}(t)\right) dt .
\end{equation*}
\end{proposition}
\begin{proof}
  Let us summarize consequences of the assumption $\bm{c}^N
  \stackrel{\tau}{\to} \mathbb{{C}}$. By Definition~\ref{def:tauConv}, this means
  $L^N_\# \bm{c}^N(t) \rightharpoonup \CC(t)$ for all $t\in [0,T]$. For $x,y\in \cX$ we define two
  auxiliary measures
  $\mathbb{\Gamma}^{N;x,y;1}(t),\mathbb{\Gamma}^{N;x,y;2}(t)\in\cP\big(\cP(\cX)\times\cP(\cX)\big)$ by setting
\begin{align*}
  \mathbb{\Gamma}^{N;x,y;1}(t,\nu,\mu) :=  \delta_{\nu^{N;x,y}}(\mu) \nu_x Q_{xy}^N(\nu) L_\#^N \bm c^N(t,\nu) \\
   \mathbb{\Gamma}^{N;x,y;2}(t,\nu,\mu) :=  \delta_{\mu^{N;y,x}}(\nu) \mu_y Q_{yx}^N(\mu) L_\#^N \bm c^N(t,\mu),
\end{align*}
where  $\nu^{N;x,y} := \nu - \frac{\delta_x - \delta_y}{N}$. Then, we have $\mathbb{\Gamma}^{N;x,y;1}(t,\nu,\mu)=\mathbb{\Gamma}^{N;y,x;2}(t,\mu,\nu)$. Due to~\eqref{ass:UniConvGammaN} from Lemma~\ref{lem:UniConvGammaN} it holds
\begin{align}
 \mathbb{\Gamma}^{N;x,y;1}(t,\nu,\mu) \rightharpoonup \delta_{\nu}(\mu) \nu_x Q_{xy}(\nu) \mathbb{C}(t,\nu) \label{lscA:CC1:weakto}\\
  \mathbb{\Gamma}^{N;x,y;2}(t,\nu,\mu) \rightharpoonup \delta_{\mu}(\nu) \mu_x Q_{xy}(\mu) \mathbb{C}(t,\mu) \label{lscA:CC2:weakto}.
\end{align}
In the sequel, we will decompose the sum over all possible jumps of the particle system in different ways
\begin{equation*}
  \sum_{\bx,\by} f(\bx,\by) = \sum_{\nu,\mu\in \cP_N(\cX)} \sum_{\substack{\bx: L^N \bx=\nu \\ \by : L^N \by=\mu}}  f(\bx,\by) =  \sum_{x,y\in \cX} \sum_{\nu\in \cP_N(\cX)}  \sum_{i=1}^N \sum_{\substack{\bx: L^N \bx = \nu \\ x_i = x}} f(\bx, \bx^{i;y}) .
\end{equation*}
where $\bx^{i;y} = \bx - (x_{i}-y)\be^i$ and $f: \cX^N \times \cX^N \to
\R$ with $f(\bx,\by)= 0$ unless $\by= \bx^{i;y}$ for some $i\in
\set{1,\dots, N}$ and $y\in \cX$.  We define the following
vector field on $\cP(\cX)\times \cP(\cX)$ 
\begin{align*}
  \vv^{N;x,y}(t,\nu,\mu) &:= \frac{1}{2N} \delta_{\nu^{N;x,y}}(\mu) \sum_{\substack{\bx: L^N \bx=\nu \\ \by : L^N \by=\mu}}\bra{\bm\psi_{\by}^N(t)-\bm\psi_{\bx}^N(t)} \bm w_{\bx\by}^N\bra{\bm c^{N}(t)} \\
  &= \frac{1}{2N} \delta_{\nu^{N;x,y}}(\mu) \sum_{i=1}^N \sum_{\substack{\bx: L^N\bx =\nu \\ x_i = x}} \bra{\bm\psi_{\bx^{i;y}}^N(t)-\bm\psi_{\bx}^N(t)} \bm w_{\bx\bx^{i;y}}^N\bra{\bm c^{N}(t)},
\end{align*}
where $\bm w_{\bx\by}^N\bra{\bm c^N(t)}$ is defined in~\eqref{e:PS:weights}.
From the definition of $\vv^{N;x,y}(t)$ and the Cauchy-Schwarz inequality, it follows that for $\nu,\mu\in \cP(\cX)$ with $\mu = \nu^{N;x,y}$ for some $x,y\in \cX$
\begin{align*}
  \abs{\vv^{N;x,y}(t,\nu,\mu)} &\leq \Biggl(\frac{1}{2N} \sum_{\substack{\bx: L^N \bx=\nu \\ \by : L^N \by=\mu}}\bra{\bm\psi_{\by}^N(t)-\bm\psi_{\bx}^N(t)}^2 \bm{w}_{\bx\by}^N\bra{\bm c^{N}(t)} \Biggr)^\frac12 \times \\
  &\qquad  \Biggl(\frac{1}{2N} \sum_{i=1}^N \sum_{\substack{\bx: L^N\bx =\nu \\ x_i = x}}\bm w^{N}_{\bx \bx^{i;y}}\bra{\bm c^N(t)}\Biggr)^\frac12.
\end{align*}
By using the identity
\[
 \frac{1}{N} \sum_{i=1}^N \sum_{\substack{\bx : L^N(\bx)=\nu \\ x_i = x}} \bm c_{\bx}^N(t) = \sum_{\bx} \delta_\nu(L^N(\bx)) \; L_{x}^N(\bx) \; \bm{{c}}_{\bx}^N(t) = L^N_{\#} \bm{{c}}^{N}(t,\nu) \ \nu_{x} ,
\]
and the fact that the logarithmic mean is jointly concave and $1$-homogeneous, we can conclude
\begin{align*}
  \frac{1}{N} \sum_{i=1}^N \sum_{\substack{\bx: L^N\bx =\nu \\ x_i = x}}\bm w_{\bx\bx^{i;y}}^N\bra{ \bm c^N(t)} \leq  \Lambda\!\bra{ \mathbb{\Gamma}^{N;x,y;1}(t,\nu,\nu^{N;x,y}),\mathbb{\Gamma}^{N;x,y;2}(t,\nu,\nu^{N;x,y})} ,
\end{align*}
which first shows that $\vv^{N;x,y}(t) \ll \Lambda\!\bra{ \mathbb{\Gamma}^{N;x,y;1}(t),\mathbb{\Gamma}^{N;x,y;2}(t) }$ as product measure on $\cP(\cX)\times \cP(\cX)$. Moreover, by summation and integration over any Borel $I\subset [0,T]$ we get
\begin{align}\label{lscA:vv:bound}
  \int_I \sum_{\nu,\mu\in \cP_N(\cX)} &\abs{\vv^{N;x,y}(t,\nu,\mu)} \;dt \leq \bra{ \sqrt{T} \int_0^T \frac{1}{N} \bm\cA(\bm c^N(t),\bm\psi(t)) \; dt }^{\frac12} \times \\
   &\bra{ \frac{\abs{I}}{2} \sum_{{\nu,\mu\in\cP_N(\cX)}} \sup_{t\in I} \Lambda\big(\mathbb{\Gamma}^{N;x,y;1}(t,\nu,\mu),\mathbb{\Gamma}^{N;x,y;2}(t,\nu,\mu)\big)}^{\frac12} . \notag
\end{align}
The second sum is uniformly bounded in $N$, since $\cX$ is finite and by Lemma~\ref{lem:UniConvGammaN} $Q^N \to Q$ uniformly with $Q$ continuous in the first argument on the compact space $\cP(\cX)$.
Now, from~\eqref{lscA:vv:bound}, we conclude that for some subsequence and all $x,y\in \cX$ we have
$\vv^{N;x,y} \rightharpoonup \vv^{x,y}$  with $\vv^{x,y}$ a Borel measure on $[0,T] \times \cP(\cX)\times \cP(\cX)$.
Using Jensen's inequality applied to the $1$-homogeneous jointly convex function $\R \times \R_+^2 \ni (v,a,b)\mapsto \frac{v^2}{\Lambda(a,b)}$, we get
\begin{align*}
  &\int_0^T \frac{1}{N}\bm\cA\bra{\bm c^N(t), \bm\psi^N(t)} \; dt\\
  &= \int_0^T \frac{1}{2} \sum_{x,y} \sum_{\nu\in \cP_N(\cX)} \frac{1}{N} \sum_{i=1}^N \sum_{\substack{\bx: L^N\bx =\nu \\ x_i = x}}\frac{\bra{ \bra{\bm\psi_{\bx^{i;y}}^{N}(t) - \bm\psi_{\bx}^{N}(t)} \bm w_{\bx\bx^{i;y}}^{N}(\bm{c}^{N}(t))}^2}{\bm{w}_{\bx\bx^{i;y}}^{N}(\bm{c}^{N}(t))} \; dt \\
  &\geq  \int_0^T \frac{1}{2} \sum_{x,y}   \sum_{\nu\in \cP_N(\cX)} \frac{\bra{\vv^N(t,\nu,\nu^{N;x,y})}^2}{\Lambda(\mathbb{\Gamma}^{N;x,y;1}(t,\nu,\nu^{N;x,y}) ,\mathbb{\Gamma}^{N;x,y;2}(t,\nu,\nu^{N;x,y}))} \; dt\;.
\end{align*}
The last term can be written as
\begin{align*}
   \frac{1}{2} \sum_{x,y} F(\vv^{N;x,y},\mathbb{\Gamma}^{N;x,y;1},\mathbb{\Gamma}^{N;x,y;2})\;,
 \end{align*}
where the functional $F$ on triples of measure on $[0,T]\times\cP(\cX)^2$ is defined via
 \begin{align*}
F(\vv,\mathbb{\Gamma}^1,\mathbb{\Gamma}^2) := \int_0^T\iint_{\cP(\cX)^2}\alpha\!\bra{\frac{d\vv}{d\sigma},\Lambda\!\bra{\frac{d\mathbb{\Gamma}^1}{d\sigma},\frac{d\mathbb{\Gamma}^2}{d\sigma}}} \, d\sigma\, dt \; ,
\end{align*}
with $\alpha$ the function defined in \eqref{e:def:alpha} and
$\sigma$ is any measure on $[0,T]\times\cP(\cX)^2$ such that
$\vv,\mathbb{\Gamma}^1,\mathbb{\Gamma}^2\ll\sigma$. The definition
does not depend on the choice of $\sigma$ by the $1$-homogeneity of
$\alpha$ and $\Lambda$.
Then, by a general result on lower semicontinuity of integral
functionals \cite[Thm.~3.4.3]{But89} we can conclude, that
\begin{align*}
  \liminf_{N\to\infty} F(\vv^{N;x,y},\mathbb{\Gamma}^{N;x,y;1},\mathbb{\Gamma}^{N;x,y;2})\geq F(\vv^{x,y},\mathbb{\Gamma}^{x,y;1},\mathbb{\Gamma}^{x,y;2})\;.
\end{align*}
In particular, this implies
\[
  d\vv^{x,y} \ll \Lambda\!\bra{\frac{d\mathbb{\Gamma}^{x,y;1}}{d\sigma} ,\frac{d\mathbb{\Gamma}^{x,y;2}}{d\sigma}} d\sigma,
\]
which by \eqref{lscA:CC1:weakto} and ~\eqref{lscA:CC2:weakto} is given in terms of
\[
  \Lambda\!\bra{\frac{d\mathbb{\Gamma}^{x,y;1}}{d\sigma}(t,\nu,\mu) ,\frac{d\mathbb{\Gamma}^{x,y;2}}{d\sigma}(t,\nu,\mu)} d\sigma= \delta_{\nu}(\mu)  \Lambda\!\bra{\nu_x Q_{xy}(\nu), \nu_y Q_{yx}(\nu)} \CC(t,d\nu)d t.
\]
Therefore, with the notation of Proposition~\ref{prop:Lio:GradientFields}, we obtain the statement
\begin{align*}
  &\liminf_{N\to\infty} \int_0^T \frac{1}{N}\bm\cA\bra{\bm c^N(t), \bm\psi(t)^N} \; dt \\
  &\geq \frac{1}{2} \sum_{x,y}  \int_0^T \int_{\cP(\cX)} \frac{\bra{\VV_{xy}(t,\nu)}^2}{\Lambda(\nu_x Q_{xy}(\nu), \nu_y Q_{yx}(\nu)) } \CC(t,d\nu) \; dt \\
  &= \int_0^T \vec\AA(\CC(t),\VV(t)) \; dt \qquad\text{with}\qquad \VV_{xy}(t,\nu) := \frac{d\vv^{x,y}}{d\CC(t)dt } .
\end{align*}
From the convergence of the vector field $\vv^{N;x,y} \rightharpoonup \vv^{x,y}$ it is straightforward to check that $(\CC,\VV)\in \vec\CCEE_T(\CC_0,\CC_T)$ and hence by the conclusion of Proposition~\ref{prop:Lio:GradientFields}, there exists $\PPsi:[0,T]\times \cP(\cX)\to \R^{\cX}$ such that
\begin{align*}
 \liminf_{N\to\infty} \frac{1}{N}\int_0^T \bm\cA\bra{\bm c^N(t), \bm\psi^N(t)}  dt \geq \int_0^T \vec\AA(\CC(t),\VV(t)) \; dt \geq \int_0^T \AA(\CC(t),\PPsi(t)) \; dt ,
\end{align*}
which concludes the first part.

The $\liminf$ estimate of the Fisher information follows by a similar
but simpler argument. The convex $1$-homogeneous function
$\lambda(a,b)=(a-b)\bra{\log a - \log b}$ allows to rewrite
\begin{align*}
   &\frac{1}{N} \bm{\cI}^{N}\left(\bm{{c}}^{N}(t)\right) dt \\
   &= \frac{1}{2} \sum_{x,y} \sum_{\nu\in \cP_N(\cX)} \frac{1}{N} \sum_{i=1}^N \sum_{\substack{\bx: L^N\bx =\nu \\ x_i = x}} \lambda\!\bra{\bm{c}_{\bx}^N(t) \bm{Q}_{\bx\bx^{i;y}}^{N}(\bm{c}^{N}(t)) , \bm{c}_{\bx^{i;y}}^{N}(t) \bm Q^N_{\bx^{i;y} \bx}(\bm{c}^{N}(t)) } \\
   &\geq \frac{1}{2} \sum_{x,y} \iint_{\cP(\cX)^2} \lambda\!\bra{ \frac{d\mathbb{\Gamma}^{N;x,y;1}(t)}{d \sigma},\frac{d\mathbb{\Gamma}^{N;x,y;2}(t)}{d\sigma}} d\sigma .
\end{align*}
Then, the result follows by an application of \cite[Thm.~3.4.3]{But89}.
\end{proof}
In order to apply Theorem \ref{thm_serfaty}, we still need to prove
that a sequence of $N$-particle dynamics starting from nice initial
conditions is tight.
\begin{lemma}\label{lem:SkorokhodTight}
  Let $\mathbf{X}^N$ be the continuous Markov jump process with
  generator~\eqref{e:def:PartGenerator}, then the sequence of laws of empirical
  measures is tight in the Skorokhod topology, i.e.\@ it holds for any
  $x \in \cX$ and $\eps>0$
\begin{equation} \label{tightness_estimate}
\underset{\delta \to 0}{\lim} \; \underset{N \to \infty}{\limsup} \; \mathbb{P}\left[\underset{|t-s|\leq \delta}{\sup} \left|L_{x}^N\!(\mathbf{X}^N\!(t)) - L_{x}^N\!(\mathbf{X}^{N}\!(s))\right| > \epsilon\right] = 0.
\end{equation}
\end{lemma}
The proof follows standard arguments for tightness of empirical
measures of sequences of interacting particle systems. Since there is
no original argument here, the exposition shall be brief, and we refer
to \cite{KipnisLandim} for more details, in a more general context of
interacting particle systems. For example, see the first step in the proof of Theorem 2.1 in \cite{KipnisLandim} for those arguments in the context of the simple exclusion process on a discrete torus.
\begin{proof}
The process
\[\begin{split}
M^N_{x}(t) &= L_{x}^N\!(\mathbf{X}^{N}\!(t)) - L_{x}^N\!(\mathbf{X}^{N}\!(0)) \\
&\quad - \int_0^t{\underset{y \neq x}{\sum} \; L_{x}^N\!(\mathbf{X}^{N}\!(s))\,Q_{xy}^N(L^N\!(\mathbf{X}^{N}\!(s))) - L_{y}^N\!(\mathbf{X}^{N}\!(s))\,Q_{yx}^N(L^N\!(\mathbf{X}^{N}\!(s))) \, ds}
  \end{split}\]
is a martingale. Since the rates are bounded,
\[
  \bigg| \int_s^t{\underset{y \neq x}{\sum} \; L_{x}^N\!(\mathbf{X}^{N}\!(r))\,Q_{xy}^N(L^N\!(\mathbf{X}^{N}\!(r))) - L_{y}^N\!(\mathbf{X}^{N}\!(r))\,Q_{yx}^N(L^N\!(\mathbf{X}^{N}\!(r)))\,dr}\bigg| \leq C|t -s|
\]
and therefore, to prove \eqref{tightness_estimate}, it is enough to show that
\[
  \mathbb{P}\!\left[\underset{|t-s|\leq \delta}{\sup} \hspace{1mm} \left|M^N_{x}(t) - M_{x}^N(s)\right| > \epsilon \right]
\]
is small. To do so, we shall estimate the quadratic variation of the martingale. It is given by
\begin{align*}
 &\langle M^N \rangle(t) = \frac{1}{N^2}\underset{y\neq x}{\sum}\big((NL_{x}^N\!(\mathbf{X}^{N}\!(t)) - 1)^2 - (NL_{x}^N\!(\mathbf{X}^{N}\!(t)))^2\big)\,Q_{xy}^N(L^N\!(\mathbf{X}^{N}\!(s))) + \\
&\Big((NL_{x}^N\!(\mathbf{X}^{N}\!(t))+ 1\big)^2 - \big(NL_{x}^N\!(\mathbf{X}^{N}\!(t)\big)^2\Big)\,Q_{yx}^N(L^N\!(\mathbf{X}^{N}\!(s)))+ \\
&\frac{2}{N^2}\underset{y \neq x}{\sum} NL_{x}^N\!(\mathbf{X}^{N}(t))\Big( L_{x}^N\!(\mathbf{X}^{N}\!(s))\,Q_{xy}^N\big(L^N\!(\mathbf{X}^{N}\!(s))\big) - L_{y}^N\!(\mathbf{X}^{N}\!(s))\,Q_{yx}^N\big(L^N\!(\mathbf{X}^{N}\!(s))\big)\Big).
\end{align*}
Using the boundedness of the rates, it is straightforward to see that
$$|\langle M^N \rangle(t)| \leq CN^{-1}.$$
The quadratic variation of the martingale vanishes. From Doob's martingale inequality, we deduce that for any $\epsilon > 0$
$$\mathbb{P}\left[\underset{0 \leq s \leq t \leq T}{\sup} |M^N_{x}(t) - M^N_{x}(s)| > \epsilon \right] \longrightarrow 0.$$
Tightness of the sequence of processes follows.
\end{proof}
We can now combine the work done so far to obtain our main result.
\begin{theorem}[Convergence of the particle system to the mean field equation]
  \label{thm:PStoMF}
  Let $\mathbf{c}^N$ be a solution to \eqref{e:def:PartEvol}.
  Moreover assume its initial distribution to be well-prepared
  \begin{equation*}
    \frac{1}{N} \bm\cF^{N}(\mathbf{c}^N(0)) \to \FF(\CC(0)) - \cF_0
    \qquad\text{with}\qquad  L^N_{\#} \mathbf{c}^N(0) \rightharpoonup \CC(0)
    \qquad\text{as } N\to \infty.
  \end{equation*}
  Then it holds
  \begin{equation*}
    L^N_\# \mathbf{c}^N(t) \rightharpoonup \CC(t) \qquad\text{ for all }t\in(0,\infty)\;,
  \end{equation*}
  with $\CC$ a weak solution to~\eqref{e:Liouville} and moreover
  \begin{equation}\label{e:PStoMF:Energy}
    \frac{1}{N}\bm\cF^{N}(\mathbf{c}^N(t))
    \to \FF(\CC(t)) - \cF_0 \qquad\text{for all~$t\in (0,\infty)$} .
  \end{equation}
\end{theorem}
\begin{proof}
  Fix $T>0$. By the tightness Lemma~\ref{lem:SkorokhodTight}, we have
  that the sequence of empirical measures $L^N_\# \mathbf{c}^N:
  [0,T] \to \cP(\cP(\cX))$ is tight w.r.t.~the Skorokhod topology
  \cite[Theorem 13.2]{Billingsley2012}. Hence, there exist a
  measurable curve $\CC: [0,T]\to \cP(\cP(\cX))$ such that up to a
  subsequence $L^N_\# \mathbf{c}^N(t)$ weakly converges to
  $\CC(t)$ for all $t\geq0$. By the Propositions~\ref{prop:liminf:Ent}
  and~\ref{prop:liminf:MDslopes}, we get from
  Theorem~\ref{thm_serfaty} that~\eqref{e:PStoMF:Energy} holds and
  $\CC$ is curve of maximal slope for the functional $\JJ$.  By
  Proposition~\ref{prop:Lio:GF}, it is characterized as weak solution
  to~\eqref{e:Liouville}. By Lemma~\ref{lem:UniConvGammaN} the
  limiting rate matrix $Q$ is Lipschitz on $\cP(\cX)$ providing uniqueness of the Liouville equation~\eqref{e:Liouville}. Hence, the convergence actually holds for the full sequence.
\end{proof}
\begin{corollary}
  In the setting of Theorem~\ref{thm:PStoMF} assume in addition that
  \begin{equation*}
  L_{\#}^N \bm{c}^N(0) \rightharpoonup \delta_{c(0)} \qquad \text{for some }\qquad c(0) \in \cP(\cX)\;.
  \end{equation*}
    Then it holds
  \begin{equation*}
  L^N_\# \mathbf{c}^N(t) \rightharpoonup \delta_{c(t)} \qquad\text{ for all }t\in(0,\infty)\;,
  \end{equation*}
  with $c$ a solution to~\eqref{eq_evol_nonlin} and moreover
  \begin{equation*}
    \frac{1}{N}\bm\cF^{N}(\mathbf{c}^N(t))
    \to \cF(c(t)) - \cF_0 \qquad\text{for all~}t\in (0,\infty)\;.
  \end{equation*}
\end{corollary}
\begin{proof}
  The proof is a direct application of Theorem~\ref{thm:PStoMF} and a variance estimate
  for the particle system (Lemma~\ref{lem:VarEst}).
\end{proof}
\section{Properties of the metric\texorpdfstring{ $\cW$}{}}\label{S:metric}
In this section, we give the proof of Propostion~\ref{prop_metric}
stating that $\cW$ defines a distance on $\cP(\cX)$ and that the
resulting metric space is seperable, complete and geodesic. The proof
will be accomplished by a sequence of lemmas giving various estimates
on and properties of $\cW$. Some work is needed in particular to show
finiteness of $\cW$.
\begin{lemma}\label{lemma:sqrt}
  For $\mu,\nu,$ and $T>0$ we have
  \begin{equation*}\begin{split} \cW(\mu, \nu) &= \inf \set{ \int_0^T
        \sqrt{\cA({c}(t), \psi(t))} \; dt : ({c},\psi)\in
        \CE_{T}(\mu,\nu)}
      .\end{split}\end{equation*}
\end{lemma}
\begin{proof}
  This follows from a standard reparametrization
  argument. See for instance \cite[Thm.~5.4]{Dolbeault2008} for
  details in a similar situation.
\end{proof}
From the previous lemma we easily deduce the
triangular inequality
\begin{align}\label{eq:triangular}
  \cW(\mu,\eta) \leq \cW(\mu,\nu) + \cW(\nu,\eta)\quad\forall \mu,\nu,\eta\in\cP(\cX)\;,
\end{align}
by concatenating two curves $(c,\psi)\in\CE_T(\mu,\nu)$ and
$(c',\psi')\in\CE_T(\nu,\eta)$ to form a curve in
$\CE_{2T}(\mu,\eta)$.

The next sequence of lemmas puts $\cW$ in relation
with the total variation distance on $\cP(\cX)$. To proceed, we
define similarly to \cite{Maas2011}, for every $\mu\in\cP(\cX)$, the
matrix
\begin{equation*}
  B_{xy}(\mu):=\begin{cases} \sum_{z\neq x}w_{xz}(\mu), & x=y,\\-w_{xy}(\mu), & x\ne y\end{cases}
\end{equation*}
Now \eqref{e:def:W} can be rewritten as
\begin{equation*}
  \cW^{2}(\mu, \nu) = \inf \set{ \int_0^1 \skp{B({c}(t))\psi(t)}{\psi(t)} \; dt
  : ({c},\psi)\in  \CE_{1}(\mu,\nu)} ,
\end{equation*}
where $\skp{\psi}{\phi} = \sum_{x\in \cX} \psi_x \phi_x$ is the usual inner product on $\R^\cX$.
\begin{lemma}\label{lemma:bounds}
  For any $\mu,\nu \in \cP(\cX)$ holds
   \[\cW(\mu,\nu)\geq \frac{1}{\sqrt{2}} \|\mu-\nu\|.\]
  Moreover, for every $a>0,$ there exists a constant $C_{a},$ such that for all $\mu,\nu\in\cP^a(\cX)$
  \begin{align*}
    \cW(\mu,\nu)\leq C_{a} \|\mu-\nu\|\;.
  \end{align*}
\end{lemma}
\begin{proof}
  The proof of the lower bound on $\cW$ can be obtained very similar
  to~\cite[Proposition 2.9]{EM11}.

  Let us show the upper bound. Following Lemma A.1. in
  \cite{Maas2011}, we notice that for $\mu\in\cP^{a}(\cX),$ the map
  $\psi \mapsto B(\mu)\psi$ has an image of dimension $d-1$. In
  addition, the dimension of the space $\{a_x : \sum_{x\in\cX} a_x = 0
  \}$ is $d-1$, therefore the map is surjective.  From the above we
  get that the matrix $B(\mu)$ restricts to an isomorphism
  $\tilde{B}(\mu),$ on the $d-1$ dimensional space $\{a_{x}:\sum
  a_{x}=0\}.$ Now, since the mapping
  $\cP^{a}(\cX)\ni\mu\rightarrow\|\tilde{B}^{-1}(\mu)\|,$ is
  continuous with respect to the euclidean metric, we have an upper
  bound $\frac{1}{c}$ by compactness. Also
  $\cP^{a}(\cX)\ni\mu\rightarrow\|\tilde{B}(\mu)\|$ has an upper bound
  $C$ as a result of all entries in $B(\mu)$ being uniformly
  bounded. From that we get
  \begin{equation*}
    c\|\psi\|\leq\|B(\mu)\psi\|\leq C\|\psi\|, \forall \mu\in\cP^{a}(\cX)
  \end{equation*}
  for some suitable positive constants.

  Similarly to the proof of Lemma 3.19 in \cite{Maas2011}, for
  $t\in[0,1],$ we set ${c}(t)=(1-t)\mu +t \nu$ and note that
  ${c}(t)$ lies in $\cP^{a}(\cX),$ since it is a convex set. Since
  $\dot{{c}}(t)=\nu- \mu \in \operatorname{Ran} B({c}(t)),$ there exists a
  unique element $\psi(t)$ for which we have
  $\dot{{c}}(t)=B({c}(t))\psi(t),$ and $\|\psi(t)\|\leq
  \frac{1}{c}\|\mu-\nu\|$.

  From that we get
  \[\cW^2(\mu,\nu)\leq\int_{0}^{1}\skp{B({c}(t))\psi(t)}{\psi(t)}\; dt\leq\frac{1}{c^{2}}C\|\mu-\nu\|^{2}.\]\qedhere
\end{proof}
\begin{lemma}\label{lemma:connection}
  For every $\mu\in\cP(\cX), \epsilon>0$ there exists $\delta>0$ such
  that $\cW(\mu,\nu)<\epsilon,$ for every $\nu\in\cP(\cX),$ with
  $\|\mu-\nu\|\leq \delta$
\end{lemma}
The proof of this lemma is similar to the proof \cite[Theorem
3.12]{Maas2011} and uses comparison of $\cW$ to corresponding quantity
on the two point space $\cX=\{a,b\}$. However, significantly more care
is needed in the present setting to implement this argument. The
reason being that the set of pairs of points $x,y$ with
$Q_{x,y}(\mu)>0$ now depends on $\mu$.
\begin{proof}
Let $\epsilon>0$ and  $\mu\in \cP(\cX)$ be fixed. Since $\cX$ is finite, it holds with $E_\mu$ defined in~\eqref{e:def:allowsTrans}
\[
  \inf\set{ Q_{xy}(\mu) : (x,y) \in E_\mu } = a > 0 .
\]
Let $B_r(\mu) = \set{\nu\in \cP(\cX) : \| \nu - \mu \| < r}$ denote a
$r$-neighborhood around $\mu$. Since, $Q(\mu)$ is continuous in $\mu$,
there exists for $\eta>0$ a $\delta_{1} > 0$ s.t.
\[
  \forall \nu \in B_{\delta_{1}}(\mu) \quad\text{holds}\quad |Q(\nu)-Q(\mu)|_{L^\infty(\cX\times\cX)} \leq \eta .
\]
Especially, it holds by choosing $\eta\leq a/2$ that $E_\mu \subseteq E_\nu$ and in addition
\[
  \inf\set{ Q_{xy}(\nu) : (x,y) \in E_\mu, \nu\in B_{\delta_{1}}(\mu) } \geq a/2.
\]
For the next argument, observe that by the concavity of the
logarithmic mean and a first order Taylor expansion holds for $a,b ,
s,t, \eta >0$
\[\begin{split}
  \Lambda((s+\eta) a , (t+\eta) b ) &\leq \Lambda(s a , t b) + \eta\bra{ \partial_s \Lambda(s a, tb) + \partial_t \Lambda(s a, t b)} \\
  &= \Lambda(s a , t b) + \eta \bra{a \Lambda_1(s a , tb) + b \Lambda_2(s a, tb)},
\end{split}\]
where $\Lambda_i$ is the $i$-th partial derivative of $\Lambda$. Therefore we can estimate for $\nu \in B_\delta(\mu).$
\[
 \begin{split}
   &\Lambda(Q_{xy}(\nu) \nu(x), Q_{yx}(\nu) \nu(y) ) - \Lambda(Q_{xy}(\mu) \nu(x), Q_{yx}(\mu) \nu(y) ) \\
   &\leq \Lambda((Q_{xy}(\mu)+\eta) \nu(x), (Q_{yx}(\mu)+\eta) \nu(y) ) - \Lambda(Q_{xy}(\mu) \nu(x), Q_{yx}(\mu) \nu(y) ) \\
   &\leq \eta \Bigl(\nu(x) \Lambda_1(Q_{xy}(\mu) \nu(x), Q_{yx}(\mu) \nu(y) ) + \nu(y)\Lambda_2(Q_{xy}(\mu) \nu(x), Q_{yx}(\mu) \nu(y) )\Bigr) \\
   &\leq \frac{2\eta}{a} \Big(Q_{xy}(\mu)\nu(x) \Lambda_1\big(Q_{xy}(\mu)\nu(x),Q_{yx}(\mu)\nu(y)\big) \\&\qquad\qquad\qquad\qquad\qquad\qquad\qquad+ Q_{yx}(\mu)\nu(y) \Lambda_2\big(Q_{xy}(\mu)\nu(x),Q_{yx}(\mu)\nu(y)\big)\Big)\\& =  \frac{2\eta}{a}  \Lambda\big(Q_{xy}(\mu)\nu(x),Q_{yx}(\mu)\nu(y)\big) \leq \Lambda\big(Q_{xy}(\mu)\nu(x),Q_{yx}(\mu)\nu(y)\big),
 \end{split}
\]
Moreover, the last identity follows directly from the one-homogeneity
of the logarithmic mean. Furthermore, we used $\eta \leq \frac{a}{2}$
to obtain the last estimate.  Repeating the argument for the other
direction we get
\begin{equation}\label{e:lem:connection:p1}
\begin{split}\frac{1}{2}\Lambda(Q_{xy}(\mu) \nu(x), Q_{yx}(\mu) \nu(y) ) & \leq \Lambda(Q_{xy}(\nu) \nu(x), Q_{yx}(\nu) \nu(y) ) \\
  & \leq 2\Lambda(Q_{xy}(\mu) \nu(x), Q_{yx}(\mu) \nu(y))
  \end{split}\end{equation} Now, let ${c}$ be an absolutely
continuous curve with respect to $\cW_{Q(\mu)}$, where $\cW_{Q(\mu)}$
is the distance that corresponds to the linear Markov process with
fixed rates $Q(\mu),$ and lives inside the ball $B_{\delta_1}(\mu)$,
then it is also absolutely continuous w.r.t.~$\cW,$ and if
$\psi$ solves the continuity equation for ${c}$, with respect to
the rates $Q(\mu)$, then there exists a $\tilde{\psi}$, that solves the
continuity equation with respect to the variable rates $Q({c}(t))$ and
 \begin{equation}\label{estabove}
  \int_0^1 \cA({c}(t),\tilde{\psi}(t)) d{t} \leq 2 \int_0^1 \cA_{Q(\mu)}({c}(t),\psi(t)) d{t} ,
\end{equation}
where $\cA_{Q(\mu)}$ is the action with fixed rate kernel
$Q(\mu)$.

Indeed let $\psi$ be a solution for the continuity equation for
${c}$ with respect to the fixed rates $Q(\mu),$
i.e.  \[\dot{{c}}_{x}(t)=\sum_{y}(\psi_{y}(t)-\psi_{x}(t))\Lambda({c}_{x}(t)Q_{xy}(\mu),{c}_{y}(t)Q_{yx}(\mu)).\]
For $(x,y)\in E_{\mu},$ and $t\in[0,1]$ we define
 \[
    \tilde{v}_{xy}(t) := (\psi_{y}(t)-\psi_{x}(t)) \Lambda({c}_{x}(t)Q_{xy}(\mu),{c}_{y}(t)Q_{yx}(\mu)) .
 \]
 Then, it is easy to verify that $({c}, \tilde v) \in
 \vec\CE({c}(0),{c}(1))$
 (cf.~Definition~\ref{def:MF:continuity_equ}) and we can estimate
 \[\begin{split}
    \int_{0}^{1} \vecfield{\cA}({c}(t),\tilde{v}(t)) \; dt &=\int_{0}^{1} \frac{1}{2}\sum_{x,y} \alpha\!\bra{\tilde{v}_{xy}(t) , \Lambda({c}_{x}(t)Q_{xy}({c}(t)),{c}_{y}(t)Q_{yx}({c}(t)))} dt\\
   &= \int_{0}^{1} \frac{1}{2} \sum_{x,y} (\psi_{y}(t)-\psi_{x}(t))^{2} \frac{\Lambda({c}_{x}(t)Q_{xy}(\mu),{c}_{y}(t)Q_{yx}(\mu))}{\Lambda({c}_{x}(t)Q_{xy}({c}(t)),{c}_{y}(t)Q_{yx}({c}(t)))}\\ &\hspace{30pt}\times\Lambda({c}_{x}(t)Q_{xy}(\mu),{c}_{y}(t)Q_{yx}(\mu)) \; dt \\
   &\stackrel{{\eqref{e:lem:connection:p1}}}{\leq}\int_{0}^{1} \frac{1}{2} \sum_{x,y} (\psi_{y}(t)-\psi_{x}(t))^{2} \; 2\Lambda({c}_{x}(t)Q_{xy}(\mu),{c}_{y}(t)Q_{yx}(\mu))dt .
 \end{split}\] Now, the existence of $\tilde{\psi}$ is a
 straightforward application of Lemma \ref{prop:Gradient
   fields}.

 Having established~\eqref{estabove}, the final result will follow by
 a comparison with the two-point space for the Wasserstein distance
 with fixed rate kernel $Q(\mu)$.

 For $\nu\in B_{\delta}(\mu)$, we can find a sequence of at most
 $(d-1)$
 measures $\mu^i\in B_{\delta}(\mu)$, such that $\mu^0=\mu$ and
 $\mu^K=\nu$ and
\[
 \operatorname{supp} \bra{\mu^i - \mu^{i-1}} = \set{x_i,y_i} \in E_\mu \qquad \text{for } i =1,\dots, K .
\]

Indeed we can use the following  matching procedure: Find a pair $(i,j)$ with $\mu_i\ne \nu_i$ and
$\mu_j\ne \nu_j$. Set $h= \min\set{|\mu_i - \nu_i| , |\mu_j -
    \nu_j|}$. Then define $\mu^1_i := \mu_i \pm h$ and $\mu^1_j :=
\mu_j \mp h$ with signs chosen as the sign of $\nu_i -
\mu_i$. After this step at least $(d-1)$-coordinates of $\mu^1$
and $\nu$ agree. This procedure finishes after at most $d-1$
steps, because the defect mass of the last pair will match.
Therewith, we can compare with the two-point space~\cite[Lemma 3.14]{Maas2011}
 \[\cW_{Q(\mu)}(\mu^{i-1},\mu^i) \leq \frac{1}{\sqrt{2p_{x_{i}y_{i}}}} \abs{\int_{1-2\mu^i_{x_i}}^{1-2\mu^{i-1} _{x_i}} \sqrt{\frac{\operatorname{arctanh} r}{r}} d{r} } \leq \frac{\delta_{1}}{2}, \]
 with  $p_{x_i y_i}=  Q_{x_i y_i}(\mu) \pi_{x_i}(\mu)$. The last estimate follows from the fact, that the function $\sqrt{\frac{\operatorname{arctanh} r}{r}} d{r}$, is integrable in $[-1,1]$. Therefore, we can find a $\delta\leq\delta_{1}$ such that for any $a,b$ with $|a-b|\leq\delta,$ we have $\int_{1-2a}^{1-2b}
 \sqrt{\frac{\operatorname{arctanh} r}{r}}
 d{r}\leq\frac{\delta_{1}}{2}\min\{1,\sqrt{2p_{x_i y_i }}\}$.
 Finally, by Lemma \ref{lemma:bounds}, we can infer that any curve has Euclidean length smaller than its action value. We can conclude that the $\cA_{Q(\mu)}$-minimizing curve between any $\mu^{i-1},\mu^{i},$ stays inside the ball $B_{\delta_{1}}(\mu),$ from which we can further conclude that
  \[\cW(\mu^{i-1},\mu^{i})\leq 2\cW_{Q(\mu)}(\mu^{i-1},\mu^i)\leq  2\frac{\delta_{1}}{2}=\delta_{1}\]
 By an application of the triangular inequality \eqref{eq:triangular}, we get $\cW(\mu,\nu)\leq (d-1)\delta_{1},$ and the proof concludes if we pick $\delta$ such that $ (d-1)\delta_{1}\leq \epsilon.$
\end{proof}
\begin{lemma}\label{lem:topology} For $\mu_{k},\mu\in\cP(\cX),$ we have
  \[\cW(\mu_{k},\mu)\rightarrow 0 \qquad\text{iff}\qquad
  \|\mu_{k}-\mu\|\rightarrow 0.\] Moreover, the space
  $\mathcal{P}(\cX),$ along with the metric $\cW,$ is a complete
  space.
\end{lemma}
\begin{proof}
  The proof is a direct application of Lemmas
  \ref{lemma:bounds} and \ref{lemma:connection}.
\end{proof}
\begin{theorem}[Compactness of curves of finite action]\label{thm:MF:compactness}
  Let $\{({c}^{k},v^{k})\}_{k},$ with \[({c}^{k},v^{k})\in
  \vecfield\CE_T(c^k(0),c^k(T)),\] be a sequence of weak solutions
  to the continuity equation with uniformly bounded action
  \begin{equation}\label{e:def:MF:actionn}
    \sup_{k\in\mathbb{N}}\left\{\int_{0}^{T}  \vecfield\cA({c}^{k}(t), v^{k}(t)) \,dt\right\} \leq C  < \infty  .
  \end{equation}
  Then, there exists a subsequence and a limit $({c},v)$, such that
  ${c}^{k}$ converges uniformly to~${c}$ in $[0,T]$,
  $({c},v)\in \vecfield\CE_T\big(c(0),c(T)\big)$ and for the action
  we have
  \begin{equation}\label{MF:action:lsc}
    \liminf_{k\to\infty} \int_0^T \vecfield\cA({c}^{k}(t), v^{k}(t)) \, dt \geq \int_0^T \vecfield\cA({c}(t), v(t)) \, dt .
  \end{equation}
\end{theorem}
\begin{proof}
  Let $x,y\in \cX$ and $(c^{k},v^{k})$ be given as in the statement.
  Using the Cauchy-Schwarz inequality, we see that for any Borel $I\subset [0,T]$ we have the a priori estimate on $v^k$
  \begin{align*}
    \int_I \frac{1}{2} \sum_{x,y} \abs{v_{xy}^k(t)} \; d{t} &\leq \int_0^T \bra{\vec\cA\bra{c^k(t),v^k(t)}}^\frac12 \bra{ \frac{1}{2} \sum_{x,y} w_{xy}\bra{c(t)} }^\frac12 dt \\
    &\leq \sqrt{ C T} \sqrt{C_w \abs{I}} ,
  \end{align*}
  with $w_{xy}\bra{c(t)}$ from~\eqref{e:MF:weights}. Since $Q$ is continuous on $\cP(\cX)$ by Definition~\ref{def:GibbsPotential},
  \[
    \sup_{\nu\in \cP(\cX)} \frac{1}{2} \sum_{x,y} w_{xy}\bra{\nu} = C_w < \infty .
  \]
  Together with the assumption~\eqref{e:def:MF:actionn}, the whole r.h.s.~is uniformly bounded in $k$. Therefore, for a subsequence holds $v^k_{xy} \rightharpoonup v_{xy}$ as Borel measure on $[0,T]$ and all $x,y\in \cX$. Now, we choose a sequence of smooth test functions $\varphi^\eps$ in \eqref{e:MF:continuity_equ:weak},
  which converge to the indicator of the interval $[t_1,t_2]$ as
  $\eps\to 0$. Therewith and using the above a priori estimate on $v^k$, we deduce
  \begin{align*}
    \abs{c^k_x(t_2) - c^k_x(t_1)} \leq \int_{t_1}^{t_2} \frac{1}{2} \sum_{y\in \cX} \bra{\abs{v_{xy}^k(t)} + \abs{v_{yx}^k(t)}} \; dt \leq \sqrt{C C_w } \sqrt{\abs{t_2-t_1}} .
  \end{align*}
  Hence, $c^{k}$ is equi-continuous and therefore converges (upto a further subsequence) to some continuous curve $c$. This, already implies that we can pass to the limit in \eqref{e:MF:continuity_equ:weak} and obtain that $(c,v)\in \CE_T$.

  Moreover, we can deduce since $\nu \mapsto Q(\nu)$ is continuous for all $x,y\in \cX$ also $c^{k}_{1;x,y} := c_x^k(t) Q_{xy}(c^k(t)) \to c_x(t) Q_{xy}(c(t))=: c_{1;x,y}(t)$ and analogue with $c^k_{2;x,y}:= c_y^k(t) Q_{yx}(c^k(t))$. We rewrite the action~\eqref{e:def:MF:action:vec} as
  \[
    \vec\cA(c^k(t),v^k(t)) = \frac{1}{2} \sum_{x,y} \alpha\!\bra{v^k_{x,y}(t), \Lambda\!\bra{c^k_{1;x,y}(t),c^k_{2;x,y}(t)}}
  \]
  The conclusion~\eqref{MF:action:lsc} follows now from \cite[Thm.~3.4.3]{But89} by noting that $(v,c_1,c_2) \mapsto \alpha\!\bra{v, \Lambda\!\bra{c_1,c_2}}$ is l.s.c., jointly convex and $1$-homogeneous and hence
  \begin{align*}
    \liminf_{k} \int_0^T \vec\cA(c^k(t),v^k(t)) \; dt &\geq \int_0^T \frac{1}{2} \sum_{x,y} \alpha\!\bra{v_{x,y}(t), \Lambda\!\bra{c_{1;x,y}(t),c_{2;x,y}(t)}} \; dt \\
    &= \int_0^T \vec\cA(c(t),v(t)) \; dt .\qedhere
  \end{align*}
\end{proof}
We can now give the proof of Proposition \ref{prop_metric}:
\begin{proof}[Proof of Proposition \ref{prop_metric}]
  Symmetry of $\cW$ is obvious, the coincidence axiom follows from
  Lemma \ref{lemma:bounds} and the triangular inequality from Lemma
  \ref{lemma:sqrt} as indicated above. The finiteness of $\cW$ comes by
  using Lemmas \ref{lemma:bounds}, \ref{lemma:connection} and the
  triangular inequality. Thus $\cW$ defines a metric. Completeness and
  separability follow directly from Lemmas \ref{lem:topology} and
  \ref{lemma:bounds}. By the direct method of the calculus of variations
  and the compactness results Proposition \ref{thm:MF:compactness}, we
  obtain for any $\mu,\nu\in\cP(\cX)$ a curve $(\gamma_t)_{t\in[0,1]}$
  with minimal action connecting them,
  i.e.~$\cW(\mu,\nu)=\int_0^1\cA(\gamma_t,\psi_t)dt=\int_0^1|\gamma'_t|^2dt$,
  where in the last equality we used Proposition \ref{accurves}. From
  this, it is easy to see that $\gamma$ is a constant speed geodesic.
\end{proof}
\appendix
\section{Stirling formula with explicit error estimate}\label{S:stirling}
\begin{lemma}\label{lem:stirling}
  Let $\nu\in \cP_N(\cX)$, then it holds
\begin{equation*}
   - \frac{\log(N+1)}{N} \leq - \frac{1}{N} \log \frac{N!}{\prod_{x\in \cX} (N\nu(x))!} - \sum_{x\in \cX} \nu(x) \log \nu(x) \leq  \frac{|\cX| \log N}{N} + \frac{1}{N} .
\end{equation*}
\end{lemma}
\begin{proof}
  We write
  \[\begin{split}
    &- \log \frac{N!}{\prod_{x\in \cX} (N\nu(x))!} = \sum_{x\in \cX} \sum_{k=1}^{N\nu(x)} \log k - \sum_{k=1}^N \log k \\
    &\geq \sum_{x\in \cX} \int_1^{N\nu(x)} \log{y} \, dy - \int_1^{N+1} \log(y) \, dy \\
    &= \sum_{x\in \cX} \Big(N\nu(x) \big(\log N\nu(x) -1 \big) -1 \Big) - (N+1) \big( \log(N+1) - 1\big)-1 \\
    &= N \sum_{x\in \cX} \nu(x) \log \nu(x) + N R_N ,
  \end{split}\]
  where the remainder $R_N$ can be estimated as follows
  \[\begin{split}
    R_N = \frac{|\cX|}{N} + \log \frac{N}{N+1} - \frac{\log(N+1)}{N} \geq - \frac{\log(N+1)}{N}
  \end{split}\]
  for $|\cX|\geq 2$ and $N\geq 1$. The other bound can be obtained by shifting the integration bounds appropriately in the above estimate.
\end{proof}
\section{Variance estimate for the particle system}\label{S:variance}
\begin{lemma}\label{lem:VarEst}
For the $N$-Particle process $\bX^{N}$ with generator~\ref{e:def:PartGenerator} holds for some $C>0$ and all $t\in [0,T]$ with $T<\infty$
\begin{equation*}
\forall x \in \cX: \qquad  \var\!\Big( L_{x}^N\!(\bX^{N}\!(t))\Big) \leq e^{C t}\Big(\var\big(L_{x}^N(\bX^{N}(0))\big) + O(N^{-1}) \Big) .
\end{equation*}
\end{lemma}
\begin{proof}
We denote with $N_x(t)= NL_{x}^N(\mathbf{X}^N\!(t))$ the empirical process of the particle number at site $x$. The empirical density process of particles at site $x$ is then $N_x(t)/N = L_{x}^N(\mathbf{X}^N\!(t))$. Therewith, we have
\begin{align*}
 \frac{d}{dt}&\var(N_x(t)) = \EX[\mathcal{L}^N N_x^2(t)] - 2\EX[N_x(t)]\EX[\mathcal{L}^N N_x(t)] \\
 &= \EX\Biggl[\underset{y}{\sum} N_x(t)  (N_x^2(t) - (N_x(t) - 1)^2)Q^{N}_{xy}(L^N({\mathbf{X}^{N}(t)}))\\
 &\phantom{=} + \underset{y}{\sum} N_y(t) (N_x^2(t) - (N_x(t) + 1)^2)Q^{N}_{yx}(L^N({\mathbf{X}^{N}(t)}))\Biggr] \\
 &\phantom{=} - 2\EX[N_x(t)]\EX\!\left[\underset{y}{\sum} N_x(t)Q^{N}_{xy}(L^N({\mathbf{X}^{N}(t)})) - N_y(t)Q^{N}_{yx}(L^N({\mathbf{X}^{N}(t)}))\right] \\
 &= 2\EX[N_x(t)^2Q_{xy}(L^N({\mathbf{X}^{N}(t)}))]- 2\EX\!\left[\underset{y}{\sum} N_x(t) N_y(t) Q^{N}_{yx}(L^N({\mathbf{X}^{N}(t)}))\right]\\
 &\phantom{=} -2\EX[N_x(t)]\EX\!\left[\underset{y}{\sum} N_x(t)Q^{N}_{xy}(L^N({\mathbf{X}^{N}(t)})) - N_y(t)Q_{yx}^{N}(L^N({\mathbf{X}^{N}(t)}))\right] +O(N) \\
 &\leq C\var(N_x(t)) + C\underset{y \neq x}{\sum}\var(N_x(t))^{1/2}\var{N_y(t)}^{1/2} \\
 &\leq  C\underset{y}{\sum} \var(N_y(t)) + O(N).
\end{align*}
In theses computations, we used the fact that $Q^{N}$ is uniformly bounded and that
the state space is finite.

Hence
\[
  \frac{d}{dt}\var(N_x(t)/N) \leq C\underset{y}{\sum} \var(N_y(t)/N) + O(N^{-1})
\]
and therefore, using Gronwall's Lemma, as soon as the sum of initial
variances goes to zero when $N$ goes to infinity, it also goes to
zero at any positive time, and uniformly on bounded time intervals.
\end{proof}

\section*{Acknowledgments}  This work was done while the authors were
enjoying the hospitality of the Hausdorff Research Institute for
Mathematics during the Junior Trimester Program on Optimal Transport,
whose support is gratefully acknowledged. We would like to thank Hong
Duong for discussions on this topic. M.F.\@ gratefully acknowledges
funding from NSF FRG grant DMS-1361185 and GdR MOMAS. M.E.\@ and A.S.\@
acknowledge support by the German Research Foundation through the
Collaborative Research Center 1060 \emph{The Mathematics of Emergent Effects}.

\def\cprime{$'$}
\providecommand{\href}[2]{#2}
\providecommand{\arxiv}[1]{\href{http://arxiv.org/abs/#1}{arXiv:#1}}
\providecommand{\url}[1]{\texttt{#1}}
\providecommand{\urlprefix}{}

\end{document}